\documentclass{amsart}
\usepackage{amssymb,amsfonts}
\usepackage[all,arc]{xy}
\usepackage{enumerate}
\usepackage{mathrsfs}
\usepackage{hyperref}
\usepackage{wasysym}
\usepackage{mathtools}
\usepackage{setspace}
\usepackage{float}
\usepackage{tikz}
\usepackage{tikz-cd}
\usepackage{makecell}
\usepackage[style=alphabetic,sorting=nyt]{biblatex}
\usepackage{multirow}
\newtheorem{thm}{Theorem}[section]
\newtheorem{cor}[thm]{Corollary}
\newtheorem{prop}[thm]{Proposition}
\newtheorem{lem}[thm]{Lemma}

\theoremstyle{definition}
\newtheorem{defn}[thm]{Definition}

\newtheorem{exmp}[thm]{Example}

\newtheorem{rem}[thm]{Remark}

\newcommand{\mycomment}[1]{}
\newcommand{\cal}{\mathcal}
\newcommand{\bb}{\mathbb}

\DeclareMathOperator{\Hom}{Hom}
\DeclareMathOperator{\Tor}{Tor}

\DeclareMathOperator{\cok}{coker}

\DeclareMathOperator{\Spec}{Spec}

\DeclareMathOperator{\Kum}{Kum}
\DeclareMathOperator{\chara}{char}
\DeclareMathOperator{\Aut}{Aut}
\DeclareMathOperator{\End}{End}

\DeclareMathOperator{\et}{\acute{e}t}
\DeclareMathOperator{\GL}{GL}
\DeclareMathOperator{\cris}{cris}

\makeatletter
\let\c@equation\c@thm
\makeatother
\numberwithin{equation}{section}

\title[The good reduction of generalized Kummer surfaces]{The good reduction of generalized Kummer surfaces in the non-supersingular case}

\author{Tianchen Zhao}
\begin{document}
\begin{abstract}
    In this article, we study the good reduction of generalized Kummer surfaces, which are K3 surfaces obtained as minimal resolutions of quotients of abelian surfaces by finite groups. In particular, we establish a criterion for good reduction when the abelian surface has non-supersingular reduction and the group is cyclic. This extends the result of Lazda and Skorobogatov on Kummer surfaces.
\end{abstract}
\maketitle

\setcounter{tocdepth}{1}
\tableofcontents

Throughout the whole article, we fix the following notations.
\[\begin{array}{ll}
    \cal{O}_K & \mbox{a Henselian DVR}\\
    K & \mbox{the fraction field of $\cal{O}_K$ with characteristic $0$}  \\
    k & \mbox{the residue field of $\cal{O}_K$, assumed to be perfect} \\
    p & \mbox{the characteristic of $k$, assumed to be positive}\\
    \overline{K} & \mbox{a fixed algebraic closure of $K$} \\
    \overline{k} & \mbox{the residue field of $\overline{K}$, also algebraically closed} \\
    G_K, G_k & \mbox{the Galois groups of $\overline{K}/K$ and $\overline{k}/k$ respectively} \\
    \omega & \mbox{a fixed primitive cubic root of unity in $\overline{K}$} \\
    i & \mbox{a fixed primitive fourth root of unity in $\overline{K}$} \\
\end{array}\]
\begin{enumerate}
    \item Let $\Gamma$ be a group acting on a set $S$. Let $F$ be a field. We denote by $F^S$ the $F$-valued permutation representation of $\Gamma$ associated to $S$.
    \item Let $n>0$ be an integer. We denote by $\underline{n}$ the trivial $G_K$-set of cardinality $n$.
\end{enumerate}

\section{Introduction}

Let $A$ be an abelian surface over $K$. Let $G$ be a finite group acting on $A$. Suppose the minimal desingularization of $A/G$ is a K3 surface. Then we denote this K3 surface by $\Kum(A,G)$. In this article, we investigate when $\Kum(A,G)$ has good reduction.

The possible isomorphism types of such $G$ are classified in \cite[Corollary 3.17]{generalized_kummer_katsura}: $G$ is isomorphic to $C_2$, $C_3$, $C_4$, $C_6$, $Q_8$, $Q_{12}$, or $\text{SL}_2(\bb{F}_3)$. By the table in \cite[Section 3]{generalized_kummer_katsura}, in the non-cyclic case, $\Kum(A,G)$ is a singular K3 surface. The moduli space of singular K3 surfaces has dimension $0$, so singular K3 surfaces only exist in isolated examples. Therefore, we focus on the cyclic case. Let $\tau\in G$ be a generator and write $\Kum(A,\tau)=\Kum(A,G)$.

By Proposition \ref{twisting argument}, if $\Kum(A,\tau)$ has good reduction, we may twist $(A,\tau)$ by a $1$-cocycle class $\phi\in H^1(G_K,\Aut(A_{\overline{K}},\tau))$ such that $\Kum(A,\tau)=\Kum(A^{\phi},\tau^{\phi})$ and $A^{\phi}$ has good reduction. Therefore, we simply assume $A$ has good reduction.

We say we are in the tame case if $p\nmid \lvert G\rvert$. Otherwise, we are in the wild case. The tame case is simple: let $\cal{A}$ be the N\'eron model of $A$. The action of $G$ extends uniquely to $\cal{A}$ and restricts to $\cal{A}_k$. The singularities of $A/G$ and $\cal{A}_k/G$ are the same. Then we may desingularize $\cal{A}/G$ to get a smooth model of $\Kum(A,G)$. Therefore, we are left with the wild case.

We abuse notation and write $\tau$ also for its extension to $\cal{A}$ as well as the restriction to $\cal{A}_k$. Since we want to consider the wild case and $\tau$ has order $2$, $3$, $4$, or $6$, we are reduced to the following cases:
\begin{enumerate}
    \item $\tau$ has order $2$ and $p=2$.
    \item $\tau$ has order $3$ and $p=3$.
    \item $\tau$ has order $4$ and $p=2$.
    \item $\tau$ has order $6$ and $p=2$.
    \item $\tau$ has order $6$ and $p=3$.
\end{enumerate}

We say $\cal{A}_k$ is ordinary, almost ordinary, or supersingular if
\[
    \lvert \cal{A}[p](\overline{k})\rvert = p^r
\]
with $r=2, 1$, and $0$ respectively. Equivalently, we say $A$ has ordinary, almost ordinary, or supersingular good reduction.

Suppose $\tau$ has order $2$. By \cite[Lemma 3.5]{generalized_kummer_katsura}, $\tau$ coincides with the inversion automorphism of $A$, so $\Kum(A,\tau)=\Kum(A)$ is the usual Kummer surface. Then Lazda and Skorobogatov gave a criterion in the non-supersingular case for $\Kum(A)$ to have good reduction in \cite{non-supersingular-kummer}.

To state their results, let $G$ be any finite flat group scheme over $\cal{O}_K$. Then we have a short exact sequence of $G_K$-modules:
\begin{equation*}
    0\to G^\circ(\overline{K})\to G(\overline{K})\overset{\rho}{\to} G(\overline{k})\to 0.
\end{equation*}
This is a variation of the connected-\'etale sequence. For more details, see Section \ref{Cartier duality and the connected etale sequence}. This short exact sequence will appear throughout this article.

\begin{thm}\label{theorem usual kummer}\cite[Theorem 2, Theorem 3]{non-supersingular-kummer}
    Suppose $p=2$. Let $A$ be an abelian surface over $K$ with good reduction.
    \begin{enumerate}
        \item Suppose $A$ has ordinary good reduction. Then $\Kum(A)$ has good reduction if and only if there is a $G_K$-equivariant splitting to the short exact sequence
        \[
            0\to \cal{A}[2]^\circ(\overline{K})\to A[2](\overline{K})\to \cal{A}[2](\overline{k})\to 0.
        \]
        \item Suppose $A$ has almost ordinary good reduction. Then $\Kum(A)$ has good reduction if and only if $A[2](\overline{K})$ is trivial.
    \end{enumerate}
\end{thm}

For the supersingular case, we know that $\Kum(A)$ has potential good reduction by \cite[Theorem 1.2]{matsumoto_supersingularkummer}, but the criterion for good reduction remains open.

In this article, we adapt the method in \cite{non-supersingular-kummer} to solve the non-supersingular cases when $\tau$ has order $3$, $4$, or $6$. In particular, in these cases, $A$ must have ordinary good reduction by Corollary \ref{ordinary}.

\begin{thm}\label{main theorem}
    Let $A$ be an abelian surface over $K$ with ordinary good reduction. Let $\tau$ be an automorphism of $A$, such that the minimal desingularization $\Kum(A,\tau)$ of $A/\tau$ is a K3 surface.
    \begin{enumerate}
        \item Suppose $\tau$ has order $3$ and $p=3$. Then $\Kum(A,\tau)$ has good reduction if and only if $\omega\in K$ and there is a $G_K$-equivariant splitting to the short exact sequence
        \[
            0\to \cal{A}[1-\tau]^\circ(\overline{K})\to A[1-\tau](\overline{K})\to \cal{A}[1-\tau](\overline{k})\to 0.
        \]
        \item Suppose $\tau$ has order $4$ and $p=2$. Then $\Kum(A,\tau)$ has good reduction if and only if $i\in K$, $A[1-\tau](\overline{K})$ is trivial, and there is a $G_K\times\langle\tau\rangle$-equivariant splitting to the short exact sequence
        \[
            0\to \cal{A}[2]^\circ(\overline{K})\to A[2](\overline{K})\to \cal{A}[2](\overline{k})\to 0.
        \]
        \item Suppose $\tau$ has order $6$ and $p=2$. Then $\Kum(A,\tau)$ has good reduction if and only if there is a $G_K$-equivariant splitting to the short exact sequence
        \[
            0\to \cal{A}[2]^\circ(\overline{K})\to A[2](\overline{K})\to \cal{A}[2](\overline{k})\to 0.
        \]
        \item Suppose $\tau$ has order $6$ and $p=3$. Then $\Kum(A,\tau)$ has good reduction if and only if $\omega\in K$ and there is a $G_K$-equivariant splitting to the short exact sequence
        \[
            0\to \cal{A}[1-\tau^2]^\circ(\overline{K})\to A[1-\tau^2](\overline{K})\to \cal{A}[1-\tau^2](\overline{k})\to 0.
        \]
    \end{enumerate}
\end{thm}

\noindent \textbf{Acknowledgements.} I would like to express my sincere gratitude to my supervisor, Christopher Lazda, for the invaluable guidance and support throughout this work. I am also grateful to Alvaro Gonzalez Hernandez for the helpful discussions on generalized Kummer surfaces in positive characteristics, and to Ken Lee for the assistance with crystalline cohomology. Last but not least, I gratefully acknowledge the financial support provided by the University of Exeter.

\section{Preliminaries}

In this section, we introduce the preliminaries of the article.

\subsection{Katsura's criterion}

We first review Katsura's criterion characterizing when the minimal desingularization of $A/G$ is a K3 surface. Let $F$ be a field and let $A$ be an abelian surface over $F$. Let $G$ be a finite group acting on $A$.

We say a curve $C\subset A$ is a fixed curve of $G$ if there exists $g\in G$ different from the identity, such that $g$ is trivial on $C$. 

\begin{thm}\cite[Theorem 2.4]{generalized_kummer_katsura}\label{Katsura's first criterion}
    Assume $\chara F\neq 2$. Let $A$ be an abelian surface over $F$ and let $G$ be a finite group acting on $A$. Then the minimal desingularization of $A/G$ is a K3 surface if and only if the following conditions hold:
    \begin{enumerate}
        \item the $G$-action has no fixed curves;
        \item there exists some $g\in G$ with isolated fixed points;
        \item all the singularities of $A/G$ are rational double points;
        \item the $G$-action is symplectic, i.e. $G$ acts trivially on $H^0(A,\omega_A)$.
    \end{enumerate}
\end{thm}

\begin{defn}
    We say the $G$-action on $A$ is rigid if for every $g\in G$ different from the identity, $g$ only has finitely many fixed points on $A$. Similarly, we say an automorphism $\tau$ of finite order of $A$ is rigid if the action of $\langle\tau\rangle$ on $A$ is rigid.
\end{defn}

We may rephrase Theorem \ref{Katsura's first criterion} as follows using \cite[Lemma 3.4]{generalized_kummer_katsura} and \cite[Proposition 3.3]{rybakov_generalized_kummer}:

\begin{thm}\label{Katsura's criterion}
    Assume $\chara F\neq 2$. Let $A$ be an abelian surface over $F$ and let $G$ be a finite group acting on $A$. Then the minimal desingularization of $A/G$ is a K3 surface if and only if the following conditions hold:
    \begin{enumerate}
        \item the $G$-action is rigid;
        \item the $G$-action is symplectic;
        \item $A/G$ is singular with rational double points.
    \end{enumerate}
\end{thm}

Now we return to the good reduction problem. Let $A$ be an abelian surface over $K$ with good reduction. Let $\tau$ be an automorphism of $A$ of order $3$, $4$, or $6$, such that the minimal desingularization $\Kum(A,\tau)$ of $A/\tau$ is a K3 surface. Suppose we are in the wild case. Let $\ell\neq p$ be a prime number.

\begin{prop}\label{min poly}
    The automorphisms induced by $\tau$ on $V_\ell(A)$ and on $V_\ell(\cal{A}_k)$ have the same characteristic polynomial and minimal polynomial:
    \begin{enumerate}
        \item $(X^2+X+1)^2$ and $X^2+X+1$ if $\tau$ has order $3$,
        \item $(X^2+1)^2$ and $X^2+1$ if $\tau$ has order $4$,
        \item $(X^2-X+1)^2$ and $X^2-X+1$ if $\tau$ has order $6$.
    \end{enumerate}
\end{prop}
\begin{proof}
    Since $A$ has good reduction, there is a $\tau$-equivariant isomorphism $V_\ell(A)\cong V_\ell(\cal{A}_k)$. This proves that $\tau$ on $V_\ell(A)$ and on $V_\ell(\cal{A}_k)$ have the same characteristic polynomial and minimal polynomial. By the table in \cite[Section 3]{generalized_kummer_katsura}, the characteristic polynomial of $\tau$ on $V_\ell(A)$ is $(X^2+X+1)^2$, $(X^2+1)^2$, and $(X^2-X+1)^2$ respectively when $\tau$ has order $3$, $4$, and $6$.

    Suppose $\tau$ has order $3$. Then its minimal polynomial divides both $X^3-1$ and $(X^2+X+1)^2$. Therefore, the minimal polynomial is $X^2+X+1$. Similarly, we obtain the minimal polynomial of $\tau$ in the remaining cases.
\end{proof}

\begin{cor}\label{key identities of tau}
    \begin{enumerate}
        \item Suppose $\tau$ has order $3$. Then
        \[
            \tau^2+\tau+1=0.
        \]
        \item Suppose $\tau$ has order $4$. Then
        \[
            \tau^2+1=0.
        \]
        \item Suppose $\tau$ has order $6$. Then
        \[
            \tau^2-\tau+1=0.
        \]
        In particular, $\tau^3=-1$ and $\tau^4+\tau^2+1=0$.
    \end{enumerate}
\end{cor}
\begin{proof}
    By \cite[Lemma 12.2]{Abelian_variety_milne_book}, the representation of $\End(A)$ on $T_\ell(A)$ is faithful. Then it follows from Proposition \ref{min poly}.
\end{proof}

\begin{cor}\label{ordinary}
    The abelian surface $A$ has either ordinary or supersingular good reduction.
\end{cor}
\begin{proof}
    We assume $\cal{A}_k$ is non-supersingular and prove that $\cal{A}_k$ is ordinary. Let $n$ be the order of $\tau$. By Theorem \ref{Katsura's criterion}, $\tau$ on $A$ is rigid. Since $A$ has good reduction, there is a $\tau$-equivariant isomorphism $V_\ell(A)\cong V_\ell(\cal{A}_k)$. By \cite[Proposition 3.3]{rybakov_generalized_kummer}, $\tau$ is also rigid on $\cal{A}_k$.

    Let $\zeta_n$ be a primitive $n$-th root of unity. We view $T_p(\cal{A}_k)$ as a module over $\bb{Z}_p[\zeta_n]$ where $\zeta_n$ acts as $\tau$. Since $\cal{A}_k$ is not supersingular, $1$ is not an eigenvalue of any non-trivial power of $\tau$ by \cite[Proposition 2.3.2]{Alvaro_thesis}. In particular, since $\bb{Z}_p[\zeta_n]$ is a DVR, $T_p(\cal{A}_k)$ is free over $\bb{Z}_p[\zeta_n]$. Therefore, $T_p(\cal{A}_k)$ has rank $2$ over $\bb{Z}_p$ and $\cal{A}_k$ is ordinary.
\end{proof}

\subsection{Cartier duality and the connected-\'etale sequence}\label{Cartier duality and the connected etale sequence}

This sub-section is based on \cite[Section 4.1]{non-supersingular-kummer}. Let $G$ be a finite flat group scheme over $\cal{O}_K$ with the connected-\'etale sequence
\[
    0\to G^\circ\to G\to G^{\et}\to 0.
\]
Since $(G_k)^{\et}=(G^{\et})_k$, we obtain the following isomorphism of $G_K$-modules:
        \[
            G^{\et}(\overline{K}) = G^{\et}(K^{\rm{nr}})\cong G^{\et}(\overline{k}) = G(\overline{k}).
        \]
Now if we take the $\overline{K}$-points of the connected-\'etale sequence, we get
\begin{equation}\label{variation of connected etale sequence}
    0\to G^\circ(\overline{K})\to G(\overline{K})\overset{\rho}{\to} G(\overline{k})\to 0,
\end{equation}
where $\rho:G(\overline{K})\to G(\overline{k})$ is the reduction map. We denote it by $\rho$ throughout this article.

Denote by $D(G)$ the Cartier dual of $G$. Then we have the following lemma:

\begin{lem}\label{group scheme embedding}
    Suppose the degree of $G$ is a power of $p$. Then the canonical isomorphism $D(D(G))\overset{\sim}{\to} G$ induces an embedding
    \[
        D(D(G)^{\et})\hookrightarrow G^\circ.
    \]
    In particular, if $D(G)^{\et}$ and $G^\circ$ have the same degree, $D(D(G)^{\et})\cong G^\circ$.
\end{lem}
\begin{proof}
    Since the degree of $G$ is a power of $p$, the same is true for $D(G)$ and $D(G)^{\et}$. Thus, $D(G)^{\et}$ is unipotent and $D(D(G)^{\et})$ is connected. Then the image of $D(D(G)^{\et})$ under $D(D(G))\overset{\sim}{\to} G$ lies in $G^\circ$.
\end{proof}

\subsubsection{Fixed points and degree calculation}\label{fixed points and degrees}

We return to the good reduction problem. Let $A$ be an abelian surface over $K$ with ordinary good reduction. Let $\tau$ be an automorphism of $A$, such that the minimal desingularization $\Kum(A,\tau)$ of $A/\tau$ is a K3 surface. Suppose we are in the wild case. We can now compute the degrees and the numbers of $\overline{K}$-points of $\cal{A}[1-\tau^j]$, $\cal{A}[1-\tau^j]^\circ$, and $\cal{A}[1-\tau^j]^{\et}$ for different $j$.

Suppose $\tau$ has order $3$ and $p=3$. Then the only non-trivial case is $j=1$. By Corollary \ref{key identities of tau}, $(1-\tau)^2=-3\tau$. Then we have $\cal{A}[1-\tau]\subset \cal{A}[3]$. Also, we obtain
\[
    (\text{deg}\cal{A}[1-\tau])^2 = \text{deg}\cal{A}[3]=81.
\]
Therefore, $\text{deg}\cal{A}[1-\tau]=\lvert A[1-\tau](\overline{K})\rvert=9$. Moreover, we have
\[
    \lvert \cal{A}[1-\tau](\overline{k})\rvert^2 = \lvert \cal{A}[3](\overline{k})\rvert=9.
\]
Thus, $\text{deg}\cal{A}[1-\tau]^{\et}=\lvert \cal{A}[1-\tau](\overline{k})\rvert=3$. Finally, using the connected-\'etale sequence and \eqref{variation of connected etale sequence}, we obtain $\text{deg}\cal{A}[1-\tau]^\circ=\lvert\cal{A}[1-\tau]^\circ(\overline{K})\rvert=3$.

Similarly, we do the computation in the other cases. The results are summarized as follows:
\[
    \begin{array}{c|c|c|c|c|c}
    \operatorname{ord}(\tau) & p & j
    & \makecell{\deg \cal{A}[1-\tau^j] \\ {}= \lvert\cal{A}[1-\tau^j](\overline{K})\rvert}
    & \makecell{\text{deg}\cal{A}[1-\tau^j]^\circ \\ {}=\lvert\cal{A}[1-\tau^j]^\circ(\overline{K})\rvert}
    & \makecell{\deg \cal{A}[1-\tau^j]^{\et} \\ {}= \lvert\cal{A}[1-\tau^j](\overline{k})\rvert}
    \\ \hline
    3 & 3 & 1 & 9 & 3 & 3 \\ \hline
    \multirow{2}{*}{4} & \multirow{2}{*}{2} & 1 & 4 & 2 & 2 \\
    & & 2 & 16 & 4 & 4 \\ \hline
    \multirow{6}{*}{6} & \multirow{3}{*}{2} & 1 & 1 & 1 & 1 \\
    & & 2 & 9 & 1 & 9 \\
    & & 3 & 16 & 4 & 4 \\ \cline{2-6}
    & \multirow{3}{*}{3} & 1 & 1 & 1 & 1 \\
    & & 2 & 9 & 3 & 3 \\
    & & 3 & 16 & 1 & 16
    \end{array}
\]
In particular, $\tau$ on $A_{\overline{K}}$ and $\cal{A}_{\overline{k}}$ only fixes the origin when $\tau$ has order $6$.

\subsubsection{An isomorphism of $G_K$-modules}

Let $A$ be an abelian surface over $K$ with ordinary good reduction. Let $\cal{A}^\vee$ be the dual of $\cal{A}$, which exists by \cite[Theorem 5, Chapter 8]{Neron_models}.

\begin{cor}\label{group scheme embedding special case}
    \begin{enumerate}
        \item Suppose $p=3$ and $\omega\in K$. Let $\tau$ be an automorphism of $A$ of order $3$, such that the minimal desingularization of $A/\tau$ is a K3 surface. Then there is an isomorphism of $G_K$-modules
        \[
            D(\cal{A}[1-\tau]^\circ)(\overline{k})\cong \cal{A}[1-\tau]^\circ(\overline{K}).
        \]
        \item Suppose $p=3$ and $\omega\in K$. Let $\tau$ be an automorphism of $A$ of order $6$, such that the minimal desingularization of $A/\tau$ is a K3 surface. Then there is an isomorphism of $G_K$-modules
        \[
            D(\cal{A}[1-\tau^2]^\circ)(\overline{k})\cong \cal{A}[1-\tau^2]^\circ(\overline{K}).
        \]
    \end{enumerate}
\end{cor}
\begin{proof}
    (2) is a consequence of (1), so we focus on (1). Let $\tau^\vee:\cal{A}^\vee\to \cal{A}^\vee$ be the dual of $\tau$. By \cite[Corollary 1.3]{oda_first_derham_cohomology}, $\cal{A}[1-\tau]$ and $\cal{A}^\vee[1-\tau^\vee]$ are Cartier dual to each other. Also, since $A$ has ordinary good reduction, $\cal{A}[1-\tau]^\circ$ and $D(D(\cal{A}[1-\tau])^{\et})$ both have degree $3$. By Lemma \ref{group scheme embedding}, $\cal{A}[1-\tau]^\circ\cong D(D(\cal{A}[1-\tau])^{\et})\cong D(\cal{A}^\vee[1-\tau^\vee]^{\et})$ and $D(\cal{A}[1-\tau]^\circ)\cong \cal{A}^\vee[1-\tau^\vee]^{\et}$. Then we have
    \begin{align*}
        \cal{A}[1-\tau]^\circ(\overline{K}) & \cong D(\cal{A}^\vee[1-\tau^\vee]^{\et})(\overline{K}) \\ & \cong \Hom(\cal{A}^\vee[1-\tau^\vee]^{\et}(\overline{K}),\bb{F}_3) \\ & \cong \Hom(\cal{A}^\vee[1-\tau^\vee](\overline{k}),\bb{F}_3),
    \end{align*}
    and
    \[
        D(\cal{A}[1-\tau]^\circ)(\overline{k}) \cong \cal{A}^\vee[1-\tau^\vee]^{\et}(\overline{k})=\cal{A}^\vee[1-\tau^\vee](\overline{k}).
    \]
    Therefore, $D(\cal{A}[1-\tau]^\circ)(\overline{k})$ and $ \cal{A}[1-\tau]^\circ(\overline{K})$ are $1$-dimensional representations of $G_K$ over $\bb{F}_3$ dual to each other, which are isomorphic.
\end{proof}

\subsection{The $\ell$-adic cohomology of generalized Kummer surfaces}

Let $A$ be an abelian surface over a perfect field $F$. Let $\tau$ be an automorphism of $A$, such that the minimal desingularization of $A/\tau$ is a K3 surface $X$. Let $\ell\neq \chara F$ be a prime number. Denote by $S$ the set of exceptional divisors of $X\to A_{\overline{F}}/\tau$.

\begin{prop}\label{cohomology computation}
    There is a Galois-equivariant isomorphism:
    \[
        H^2_{\et}(X_{\overline{F}},\bb{Q}_\ell)\cong H^2_{\et}(A_{\overline{F}},\bb{Q}_\ell)^{\tau=\textnormal{id}}\oplus \bb{Q}_\ell^S(-1).
    \]
\end{prop}

To prove this, we need the following lemma.

\begin{lem}\cite[Proposition 3.2.1]{Harder_Narasimhan_1975}\label{quotient cohomology}
    For every $j\geq 0$, there is a Galois-equivariant isomorphism:
    \[
        H^j_{\et}(A_{\overline{F}}/\tau,\bb{Q}_\ell)\cong H^j_{\et}(A_{\overline{F}},\bb{Q}_\ell)^{\tau=\textnormal{id}}.
    \]
\end{lem}

\begin{proof}[Proof of Proposition \ref{cohomology computation}]
    Let $\pi:X\to A_{\overline{F}}/\tau$ be the minimal desingularization of $A_{\overline{F}}/\tau$. This is an isomorphism away from the singularities on $A_{\overline{F}}/\tau$, on which $R^j\pi_*\bb{Q}_\ell$ is supported. The fibres over these points are exactly the exceptional divisors of $\pi$, which are configurations of smooth rational curves with dual graphs of ADE type. Thus, $R^1\pi_*\bb{Q}_\ell=0$ and $R^2\pi_*\bb{Q}_\ell$ is the $\bb{Q}_\ell(-1)$-valued skyscraper sheaf supported on the singular points of $A_{\overline{F}}/\tau$. In particular,
    \begin{equation}\label{higher pushforward cohomology}
        \begin{split}
            H^j_{\et}(A_{\overline{F}}/\tau,R^1\pi_*\bb{Q}_\ell) & = 0\text{ for all $j\geq 0$,} \\
            H^0_{\et}(A_{\overline{F}}/\tau,R^2\pi_*\bb{Q}_\ell) & = \bb{Q}_\ell^S(-1).
    \end{split}
    \end{equation}
    Consider the Leray spectral sequence 
    \[
        H^p_{\et}(A_{\overline{F}}/\tau,R^q\pi_*\bb{Q}_\ell)\implies H^{p+q}_{\et}(X_{\overline{F}},\bb{Q}_\ell).
    \]
    By \eqref{higher pushforward cohomology} and Lemma \ref{quotient cohomology}, we get the following exact sequence
    \[
        0\to H^2_{\et}(A_{\overline{F}},\bb{Q}_\ell)^{\tau=\text{id}}\to H^2_{\et}(X_{\overline{F}},\bb{Q}_\ell)\to \bb{Q}_\ell^S(-1)\to H^3_{\et}(A_{\overline{F}},\bb{Q}_\ell)^{\tau=\text{id}}.
    \]
    By Theorem \ref{Katsura's criterion}, the $\tau$-action on $A$ is rigid. In particular, by \cite[Proposition 3.3]{rybakov_generalized_kummer}, $H^1_{\et}(A_{\overline{F}},\bb{Q}_\ell)^{\tau=\text{id}}=0$, so $1$ is not an eigenvalue of $\tau$ on $H^1_{\et}(A_{\overline{F}},\bb{Q}_\ell)$. By Poincar\'e duality, $H^3_{\et}(A_{\overline{F}},\bb{Q}_\ell)$ and $H^1_{\et}(A_{\overline{F}},\bb{Q}_\ell)$ are dual to each other as representations of $\langle\tau\rangle$. This implies that $1$ is not an eigenvalue of $\tau$ on $H^3_{\et}(A_{\overline{F}},\bb{Q}_\ell)$. Therefore, $H^3_{\et}(A_{\overline{F}},\bb{Q}_\ell)^{\tau=\text{id}}=0$. We then obtain the following short exact sequence:
    \[
        0\to H^2_{\et}(A_{\overline{F}},\bb{Q}_\ell)^{\tau=\text{id}}\to H^2_{\et}(X_{\overline{F}},\bb{Q}_\ell)\to \bb{Q}_\ell^S(-1)\to 0.
    \]
    Note that the cycle class map provides a section for this short exact sequence, which finishes the proof.
\end{proof}

\subsection{The good reduction of $A$ and the arithmetic twists}

Let $A$ be an abelian surface over $K$. Let $\tau$ be an automorphism of $A$ of order $n\in \{3,4,6\}$, such that the minimal desingularization $X=\Kum(A,\tau)$ of $A/\tau$ is a K3 surface. Suppose $X$ has good reduction. In this sub-section, we prove the following:

\begin{prop}\label{twisting argument}
    There exists $\phi\in H^1(G_K,\Aut(A_{\overline{K}},\tau))$ such that $\Kum(A,\tau)=\Kum(A^{\phi},\tau^{\phi})$ and $A^{\phi}$ has good reduction.
\end{prop}

Let $\zeta_m$ be a primitive $m$-th root of unity for any integer $m>0$. Fix any prime number $\ell\neq p$ such that the following hold:
\begin{enumerate}
    \item If $n=3$ or $6$, $\ell\equiv 7\text{ mod } 12$, so that $\zeta_3\in \bb{Q}_\ell$ and $\zeta_4\notin \bb{Q}_\ell$.
    \item If $n=4$, $\ell\equiv 5\text{ mod } 12$, so that $\zeta_4\in \bb{Q}_\ell$ and $\zeta_3\notin \bb{Q}_\ell$.
\end{enumerate}
The reason for this choice will become clear in the proof of Lemma \ref{computing galois reps}. Let $\chi:G_K\to \Aut_{\bb{Q}_\ell}(H^1_{\et}(A_{\overline{K}},\bb{Q}_\ell))$ and $\alpha:\langle\tau\rangle\to \Aut_{\bb{Q}_\ell}(H^1_{\et}(A_{\overline{K}},\bb{Q}_\ell))$ respectively be the actions of $G_K$ and $\langle\tau\rangle$ on $H^1_{\et}(A_{\overline{K}},\bb{Q}_\ell)$.

\begin{lem}\label{computing galois reps}
    For every $\sigma\in I_K$, there exists $k_{\sigma}\in\bb{Z}$ such that $\chi(\sigma)=\alpha(\tau)^{k_{\sigma}}$.
\end{lem}
\begin{proof}
    Let $S_K$ be the set of exceptional divisors of the minimal desingularization $X_{\overline{K}}\to A_{\overline{K}}/\tau$. Since $X$ has good reduction, $H^2_{\et}(X_{\overline{K}},\bb{Q}_\ell)$ is unramified. Then by Proposition \ref{cohomology computation}, the absolute inertia group $I_K$ of $K$ acts trivially on $H^2_{\et}(A_{\overline{K}},\bb{Q}_\ell)^{\tau=\text{id}}$ and $S_K$.

    By Proposition \ref{min poly}, there exists a basis $e_1,f_1,e_2,f_2$ of $H^1_{\et}(A_{\overline{K}},\bb{Q}_\ell)$ such that $\tau$ acts as a diagonal matrix $M=\text{diag}(\zeta_n, \zeta_n^{-1},\zeta_n,\zeta_n^{-1})$. Therefore, $H^2_{\et}(A_{\overline{K}},\bb{Q}_\ell)^{\tau=\text{id}}$ has a basis
    \[
        \{e_1\wedge f_1, e_1\wedge f_2, e_2\wedge f_1, e_2\wedge f_2\}.
    \]
    Fix an arbitrary $\sigma\in I_K$. Since $\sigma$ acts trivially on $H^2_{\et}(A_{\overline{K}},\bb{Q}_\ell)^{\tau=\text{id}}$, $\sigma$ acts on $H^1_{\et}(A_{\overline{K}},\bb{Q}_\ell)$ as a diagonal matrix $\text{diag}(a,a^{-1},a,a^{-1})$ with respect to $e_1,f_1,e_2,f_2$ by computation.

    By Grothendieck's $\ell$-adic monodromy theorem \cite[Corollary 1.25]{Fontaine_Ouyang}, there exists a finite Galois extension $L/K$ such that every element in $I_L$ acts on $H^1_{\et}(A_{\overline{K}},\bb{Q}_\ell)$ unipotently. The only unipotent matrix of the form $\text{diag}(a,a^{-1},a,a^{-1})$ is the identity, so $I_L$ acts trivially on $H^1_{\et}(A_{\overline{K}},\bb{Q}_\ell)$. Then $A_L$ has good reduction. Let $\cal{A}$ be the N\'eron model of $A_L$ over $\cal{O}_L$. The semi-linear action of $\sigma$ on $A_L$ extends to a semi-linear action on $\cal{A}$, and restricts to a $k_L$-automorphism on $\cal{A}_{k_L}$. By the smooth proper base change theorem, there is a $\sigma$-equivariant isomorphism $H^1_{\et}(A_{\overline{K}},\bb{Q}_\ell)\cong H^1_{\et}(\cal{A}_{\overline{k}},\bb{Q}_\ell)$. By \cite[Proposition 12.9]{Abelian_variety_milne_book}, the trace of $\sigma$ on $H^1_{\et}(\cal{A}_{\overline{k}},\bb{Q}_\ell)$ is an integer. Thus, $2(a+a^{-1})\in \bb{Z}$ and $a\in \{\pm 1, \pm\zeta_3,\pm\zeta_3^2,\pm\zeta_4\}$. Then by the choice of $\ell$, we have the following:
    \begin{enumerate}
        \item if $n=3$ or $6$, $a=\pm \zeta_3^k$ for some integer $k$;
        \item if $n=4$, $a=\pm\zeta_4^k$ for some integer $k$.
    \end{enumerate}
    This finishes the proof when $n=4$ or $6$. We are left with the $n=3$ case.

    In this case, it remains to exclude the possibility that $a=-1$. Suppose $\chi(\sigma)=-\text{id}$. Using the $\sigma$-equivariant isomorphism $H^1_{\et}(A_{\overline{K}},\bb{Q}_\ell)\cong H^1_{\et}(\cal{A}_{\overline{k}},\bb{Q}_\ell)$, we see that $\sigma$ acts as $-\text{id}$ on $H^1_{\et}(\cal{A}_{\overline{k}},\bb{Q}_\ell)$. Then by \cite[Lemma 12.2]{Abelian_variety_milne_book}, $\sigma$ acts as the inversion automorphism on $\cal{A}_{k_L}$, therefore non-trivially on $\cal{A}[1-\tau](\overline{k})$.

    However, $\sigma$ acts trivially on $S_K$. In particular, it acts trivially on the singularities of $A_{\overline{K}}/\tau$, which are given by the projection of $A[1-\tau](\overline{K})$. Since the reduction map $\rho:A[1-\tau](\overline{K})\to \cal{A}[1-\tau](\overline{k})$ is $\sigma$-equivariant, $\sigma$ also acts trivially on $\cal{A}[1-\tau](\overline{k})$, a contradiction.
\end{proof}

Let $\phi:I_K\to \Aut(A_{\overline{K}},\tau)$ be the homomorphism that sends $\sigma\in I_K$ to $\tau^{-k_{\sigma}}$. Since the $G_K$-action on $A_{\overline{K}}$ commutes with $\tau$, it follows that $\phi$ defines a cocycle class in $H^1(I_K,\langle\tau\rangle)\subset H^1(I_K,\Aut(A_{\overline{K}},\tau))$. Also, $\phi$ is invariant under the action of $G_K$ on $I_K$ by conjugation. Therefore, $\phi$ lies in the $G_k$-invariant subset $H^1(I_K,\Aut(A_{\overline{K}},\tau))^{G_k=\text{id}}$.

We want to lift $\phi$ to a cocycle class in $H^1(G_K,\Aut(A_{\overline{K}},\tau))$. Again since the $G_K$-action on $A_{\overline{K}}$ commutes with $\tau$, we have $\langle\tau\rangle\cong \bb{Z}/n$ and $\phi\in H^1(I_K,\bb{Z}/n)^{G_k=\text{id}}$. Then it suffices to prove the following lemma:

\begin{lem}\label{cocycle lifting}
    The natural map $H^1(G_K,\bb{Z}/n)\to H^1(I_K,\bb{Z}/n)^{G_k=\textnormal{id}}$ is surjective.
\end{lem}
\begin{proof}
    By the Hochschild–Serre spectral sequence
    \[
        H^p(G_k,H^q(I_K,\bb{Z}/n))\implies H^{p+q}(G_K,\bb{Z}/n),
    \]
    it suffices to prove the inflation map
    \[
        f:H^2(G_k, \bb{Z}/n)\to H^2(G_K, \bb{Z}/n)
    \]
    is injective. The case when $p\nmid n$ follows from \cite[Equation (1.11)]{Grothendiec_Brauer_Group}.

    We are left with the case when $p\mid n$. We prove this case by case.
    \begin{enumerate}
        \item Suppose $p=3$ and $n=3$. Then by \cite[Lemma 59.63.1]{stacks-project}, $H^2(G_k,\bb{Z}/3)=0$ and $f$ is injective.
        \item Suppose $p=2$ and $n=4$. Again by \cite[Lemma 59.63.1]{stacks-project}, $H^2(G_k,\bb{Z}/2)=0$. Using the short exact sequence $0\to \bb{Z}/2\to \bb{Z}/4\to \bb{Z}/2\to 0$, we obtain $H^2(G_k,\bb{Z}/4)=0$ and $f$ is injective.
        \item Suppose $p=2$ and $n=6$. Then $H^2(G_k,\bb{Z}/2)=0$ and $H^2(G_k, \bb{Z}/3)\to H^2(G_K, \bb{Z}/3)$ is injective. Using the decomposition $\bb{Z}/6\cong \bb{Z}/2\times \bb{Z}/3$, we see that $f$ is injective.
        \item The case when $p=3$ and $n=6$ is completely analogous to the case when $p=2$ and $n=6$.
    \end{enumerate}
\end{proof}

\begin{proof}[Proof of Proposition \ref{twisting argument}]
    Using Lemma \ref{cocycle lifting}, we fix a lift of $\phi$ to $H^1(G_K,\Aut(A_{\overline{K}},\tau))$, which we also denote by $\phi$ by abuse of notation. We check the following:
    \begin{enumerate}
        \item $\Kum(A,\tau)=\Kum(A^{\phi},\tau^{\phi})$ since the image of $\phi$ is contained in $\langle\tau\rangle$.
        \item Let $\chi^{\phi}:G_K\to \Aut_{\bb{Q}_\ell}(H^1_{\et}(A_{\overline{K}}^\phi,\bb{Q}_\ell))$ be the action of $G_K$ on $H^1_{\et}(A_{\overline{K}}^\phi,\bb{Q}_\ell)$. Then for every $\sigma\in I_K$, we have $\chi^{\phi}(\sigma)=\alpha(\phi(\sigma))\cdot \chi(\sigma)=\text{id}$. Therefore, $H^1_{\et}(A_{\overline{K}}^{\phi},\bb{Q}_\ell)$ is unramified and $A^{\phi}$ has good reduction.
    \end{enumerate}
\end{proof}

\subsection{Canonical reduction of K3 surfaces}

In this sub-section, we review some general results on the good reduction of K3 surfaces. Let $\ell\neq p$ be a prime number. Let $X$ be a K3 surface over $K$ with good reduction. Let $\cal{X}$ be a smooth model of $X$. Then $\cal{X}_k$ is a K3 surface over $k$ since it has trivial canonical bundle and the same second Betti number as $X$.

Unlike abelian varieties, $X$ can have many non-isomorphic smooth models. However, their special fibres are always isomorphic: let $\cal{X}_1$ and $\cal{X}_2$ be two smooth models of $X$. By \cite[Proposition 4.7]{liedtke_2017_good_reduction_k3surfaces}, $\cal{X}_1$ and $\cal{X}_2$ are connected by finitely many birational maps defined away from finitely many curves on their special fibres. In particular, $\cal{X}_{1,k}$ and $\cal{X}_{2,k}$ are birational, thus isomorphic since they are K3 surfaces.

This result has the following important application. Let $X$ be a K3 surface over $K$ with good reduction after a finite and unramified extension. Let $L/K$ be an unramified extension, such that $X_L$ has a smooth model $\cal{X}$. Denote by $\cal{O}_L$ and $k_L$ respectively the ring of integers and the residue field of $L$. By passing to the Galois closure, we may assume $L/K$ to be Galois.

View $X_L$ as an open subscheme of $\cal{X}$. Then the action of the Galois group $G_{L/K}$ of $L/K$ on $X_L$ induces a rational action on $\cal{X}$. Also, this action is defined away from finitely many curves on $\cal{X}_{k_L}$ again by \cite[Proposition 4.7]{liedtke_2017_good_reduction_k3surfaces}. Therefore, it becomes regular if restricted to $\cal{X}_{k_L}$. Then the quotient $\cal{X}_{k_L}/G_{L/K}$ is a K3 surface over $k$\mycomment{why!}. This is independent of the choice of $\cal{X}$ and $L$, so we denote it by $X^\circ$.

\begin{defn}
    We call $X^\circ$ the canonical reduction of $X$.
\end{defn}

Note that $H^2_{\et}(X_{\overline{k}}^\circ,\bb{Q}_{\ell})$ has a natural $G_k$-action. We view it as a $G_K$-representation via the surjection $G_K\to G_k$.

To compute the canonical reduction, the following lemma is useful:

\begin{lem}\cite[c.f.][Lemma 6.2]{non-supersingular-kummer}\label{find the caonical reduction}
    Let $X$ be a K3 surface over $K$ with good reduction after a finite and unramified extension. Let $\cal{Y}$ be a flat projective scheme over $\cal{O}_K$ with normal fibres of dimension $2$, such that
    \begin{enumerate}
        \item the minimal desingularization of $\cal{Y}_K$ is $X$,
        \item there exists a finite Galois extension $L/K$ and a smooth model $\cal{X}$ of $X_L$ with a proper birational morphism $\cal{X}\to \cal{Y}_{\cal{O}_L}$.
    \end{enumerate}
    Then the minimal desingularization of $\cal{Y}_k$ is the canonical reduction $X^\circ$ of $X$.
\end{lem}
\begin{proof}
    Let $G_{L/K}$ and $I_{L/K}$ be the Galois group and the inertia group of $L/K$ respectively. Since $X$ has good reduction after a finite and unramified extension, $H^2_{\et}(X_{\overline{K}},\bb{Q}_\ell)$ is unramified. By \cite[Proposition 4.5]{liedtke_2017_good_reduction_k3surfaces} and 
    \cite[Proposition 5.5]{liedtke_2017_good_reduction_k3surfaces}, there exists a smooth model $\cal{X}'$ of $X_L$ where the action of $I_{L/K}$ on $X_L$ extends to $\cal{X}'$. Also, by \cite[Proposition 5.8]{liedtke_2017_good_reduction_k3surfaces}, the restriction of the action of $I_{L/K}$ on $\cal{X}_{k_L}'$ is trivial. Let $F/K$ be the maximal unramified subextension of $L/K$. Then by \cite[Corollary 5.12]{liedtke_2017_good_reduction_k3surfaces}, the quotient $\cal{X}'/I_{L/K}$ is smooth over $\cal{O}_F$ with special fibre $\cal{X}'_{k_L}$.

    Let $G_{F/K}$ be the Galois group of $F/K$. Then the action of $G_{F/K}$ on $(\cal{X}'/I_{L/K})_F$ restricts to a rational action on $\cal{X}'/I_{L/K}$, and therefore a regular action on the special fibre $(\cal{X}'/I_{L/K})_{k_F}\cong \cal{X}'_{k_L}$. By definition, $X^\circ=\cal{X}'_{k_L}/G_{F/K}$. Also, the birational map $\cal{X}\dashrightarrow \cal{X}'$ induced by the identity morphism on the generic fibre restricts to a $G_{L/K}$-isomorphism $\cal{X}_{k_L}\to \cal{X}'_{k_L}$. Therefore, $X^\circ\cong \cal{X}_{k_L}/G_{F/K}$.

    Since $X$ is the minimal desingularization of $\cal{Y}_K$, $X_L\to \cal{Y}_L$ is $G_{L/K}$-equivariant. Therefore, $\cal{X}\to \cal{Y}_{\cal{O}_L}$ and its restriction to the special fibre $\cal{X}_{k_L}\to \cal{Y}_{k_L}$ are $G_{L/K}$-equivariant. After taking the $G_{F/K}$-quotient, we get $X^\circ\to \cal{Y}_k$. In particular, $X^\circ$ is a K3 surface, so this is the minimal desingularization.
\end{proof}

The notion of canonical reduction has the following application, which will be very important for us.

\begin{thm}\cite{Chiarellotto_2019_neron_ogg_shafarevich_criterion_k3surface}\label{comparison theorem}
    Let $X$ be a K3 surface over $K$ that has good reduction after a finite unramified extension. Then $X$ has good reduction if and only if there is a $G_K$-equivariant isomorphism
    \[
        H^2_{\et}(X_{\overline{K}},\bb{Q}_{\ell})\cong H^2_{\et}(X_{\overline{k}}^\circ,\bb{Q}_{\ell}).
    \]
\end{thm}

\section{Computation of the Galois actions}\label{computation section}

Let $A$ be an abelian surface over $K$ with ordinary good reduction. Let $\tau$ be an automorphism of $A$, such that the minimal desingularization $\Kum(A,\tau)$ of $A/\tau$ is a K3 surface. Denote by $S_K$ and $S_k$ respectively the sets of exceptional divisors of the minimal desingularization of $A_{\overline{K}}/\tau$ and $\cal{A}_{\overline{k}}/\tau$. We view both of them as $G_K$-sets: $G_K$ acts on $S_K$ naturally, and the $G_K$-action on $S_k$ is given by the surjection $G_K\to G_k$ and the natural $G_k$-action on $S_k$.

To prove Theorem \ref{main theorem}, we need to compute the singularities of $A_{\overline{K}}/\tau$ and $\cal{A}_{\overline{k}}/\tau$ and how $G_K$ permutes $S_K$ and $S_k$. We summarize the results in the next proposition. We also include the results in the order $2$, $p=2$ case with $\cal{A}_k$ ordinary from \cite[Section 2]{non-supersingular-kummer}, which will be helpful in the order $4$ and order $6$ cases.

\begin{prop}\label{computation summary}
    \begin{enumerate}
        \item Suppose $\tau$ has order $2$ and $p=2$. Then the singularities of $A_{\overline{K}}/\tau$ and $\cal{A}_{\overline{k}}/\tau$ are respectively $16A_1$ and $4D_4^1$. Also, we have $G_K$-equivariant bijections
        \begin{align*}
            S_K & \cong A[2](\overline{K}),\\
            S_k & \cong \cal{A}[2]^\circ(\overline{K}) \times \cal{A}[2](\overline{k}).
        \end{align*}
        \item Suppose $\tau$ has order $3$ and $p=3$. Then the singularities of $A_{\overline{K}}/\tau$ and $\cal{A}_{\overline{k}}/\tau$ are respectively $9A_2$ and $3E_6^1$. Also, we have $G_K$-equivariant bijections
        \begin{align*}
            S_K & \cong A[1-\tau](\overline{K})\times \{\omega,\omega^2\},\\
            S_k & \cong \cal{A}[1-\tau](\overline{k})\times D(\cal{A}[1-\tau]^\circ)(\overline{k})\times \underline{2}.
        \end{align*}
        \item Suppose $\tau$ has order $4$ and $p=2$. Then the singularities of $A_{\overline{K}}/\tau$ and $\cal{A}_{\overline{k}}/\tau$ are respectively $4A_3+6A_1$ and $2E_7^3+D_4^1$. Also, we have $G_K$-equivariant bijections
        \begin{align*}
            S_K & \cong A[1-\tau](\overline{K}) \times \{1, \pm i\} \sqcup (A[2]({\overline{K}})\setminus A[1-\tau](\overline{K}))/\tau,\\
            S_k & \cong \underline{14}\sqcup (\mathcal{A}[2]^\circ(\overline{K})\times(\cal{A}[2](\overline{k})\setminus\cal{A}[1-\tau](\overline{k})))/\tau.
        \end{align*}
        \item Suppose $\tau$ has order $6$ and $p=2$. Then the singularities of $A_{\overline{K}}/\tau$ and $\cal{A}_{\overline{k}}/\tau$ are respectively $A_5+4A_2+5A_1$ and $E_6^1+4A_2+D_4^1$. Also, we have $G_K$-equivariant bijections
        \begin{align*}
            S_K & \cong \{1, \pm \omega, \pm \omega^2\} \sqcup (A[1-\tau^2](\overline{K})\setminus 0)/\tau \times \{\omega, \omega^2\} \sqcup (A[2]({\overline{K}})\setminus 0)/\tau,\\
            S_k & \cong \{1,\omega,\omega^2\}\times\underline{2}\sqcup (A[1-\tau^2](\overline{K})\setminus 0)/\tau\times \{\omega, \omega^2\}\sqcup (\cal{A}[2]^\circ(\overline{K})\times (\cal{A}[2](\overline{k})\setminus 0))/\tau.
        \end{align*}
        \item Suppose $\tau$ has order $6$ and $p=3$. Then the singularities of $A_{\overline{K}}/\tau$ and $\cal{A}_{\overline{k}}/\tau$ are respectively $A_5+4A_2+5A_1$ and $E_7^1+E_6^1+5A_1$. Also, we have $G_K$-equivariant bijections
        \begin{align*}
            S_K & \cong \{1, \pm \omega, \pm \omega^2\} \sqcup (A[1-\tau^2](\overline{K})\setminus 0)/\tau \times \{\omega, \omega^2\} \sqcup (A[2]({\overline{K}})\setminus 0)/\tau,\\
            S_k & \cong \underline{7}\sqcup (D(\cal{A}[1-\tau^2]^\circ)(\overline{k})\times(\cal{A}[1-\tau^2]
            (\overline{k})\setminus 0)
            )/\tau\times \underline{2}\sqcup (A[2](\overline{K})\setminus 0)/\tau.
        \end{align*}
    \end{enumerate}
\end{prop}

The rest of this section is devoted to the proof of Proposition \ref{computation summary}. Denote by $q:A\to A/\tau$ and $q_k:\cal{A}_k\to\cal{A}_k/\tau$ the canonical projection. Let $O$ and $ O_k$ be the origin of $A_{\overline{K}}$ and $\cal{A}_{\overline{k}}$ respectively.

Consider the points on $A_{\overline{K}}$ whose stabilizer subgroup in $\langle\tau\rangle$ is non-trivial. These are the points on $A_{\overline{K}}$ fixed by a non-trivial power of $\tau$. The singularities on $A_{\overline{K}}/\tau$ are given by their images under $q$. The same statement holds if we replace $A_{\overline{K}}$ by $\cal{A}_{\overline{k}}$ and $q$ by $q_k$.

\subsection{The order $2$, $p=2$ case}

By \cite[Lemma 3.5]{generalized_kummer_katsura}, $\tau=-1$ and $\Kum(A,\tau)=\Kum(A)$. We summarize the results from \cite[Section 2]{non-supersingular-kummer}, which prove Proposition \ref{computation summary}(1).

\subsubsection{Singularities and Galois action on the generic fibre}\label{Singularities and Galois action on the generic fibre order 2}

We now describe the singularities of $A_{\overline{K}}/\tau$ and how $G_K$ acts on $S_K$.

\begin{rem}
    The argument only requires $\tau=-1$ and that the base characteristic is different from $2$. In particular, this applies to
    \begin{enumerate}
        \item $A_{\overline{K}}/\tau^2$ in the order $4$, $p=2$ case,
        \item $A_{\overline{K}}/\tau^3$ in the order $6$, $p=2$ case,
        \item $A_{\overline{K}}/\tau^3$ and $\cal{A}_{\overline{k}}/\tau^3$ in the order $6$, $p=3$ case.
    \end{enumerate}
\end{rem}

The singularities on $A_{\overline{K}}/\tau$ are given by the images of $A[1-\tau](\overline{K})=A[2](\overline{K})$ under $q$. Therefore, there are in total sixteen singularities on $A_{\overline{K}}/\tau$. By translation, their local rings are isomorphic. In particular, these singularities have the same type, which is $A_1$ by \cite[Lemma 3.14]{generalized_kummer_katsura}. This also shows that there is a $G_K$-equivariant bijection $S_K\cong A[2](\overline{K})$.

\subsubsection{Singularities and Galois action on the special fibre}\label{Singularities and Galois action on the special fibre order 2}

We now describe the singularities of $\cal{A}_{\overline{k}}/\tau$ and how $G_k$ acts on $S_k$.

There are four singularities on $\cal{A}_{\overline{k}}/\tau$ given by the images of $\cal{A}[2](\overline{k})$ under $q_k$. By translation, these singularities have the same type, which is $D_4^1$ by \cite[Section 2]{non-supersingular-kummer}. These singularities are indexed $G_k$-equivariantly by $\cal{A}[2](\overline{k})$. By translation, it suffices to compute how $G_k$ permutes the exceptional divisors of the $D_4^1$ singularity at $q_k(O_k)$. By \cite[Proposition 2.1]{non-supersingular-kummer}, they are indexed $G_k$-equivariantly by $\cal{A}^\vee[2](\overline{k})$. By \cite[Section 4.1]{non-supersingular-kummer}, $\cal{A}^\vee[2](\overline{k})$ is isomorphic to $\cal{A}[2]^\circ(\overline{K})$. This shows that there is a $G_k$-equivariant bijection $S_k\cong \cal{A}[2]^\circ(\overline{K}) \times \cal{A}[2](\overline{k})$.

\subsection{The order $3$, $p=3$ case}

We now prove Proposition \ref{computation summary}(2).

\subsubsection{Singularities and Galois action on the generic fibre}\label{Singularities and Galois action on the generic fibre order 3}

We first describe the singularities of $A_{\overline{K}}/\tau$ and how $G_K$ acts on $S_K$.

\begin{rem}
    The argument only requires $\tau$ is symplectic, $\tau^2+\tau+1=0$, and that the base characteristic is different from $3$. In particular, this applies to
    \begin{enumerate}
        \item $A_{\overline{K}}/\tau^2$ and $\cal{A}_{\overline{k}}/\tau^2$ in the order $6$, $p=2$ case,
        \item $A_{\overline{K}}/\tau^2$ in the order $6$, $p=3$ case.
    \end{enumerate}
\end{rem}

We first show that the singularities of $A_{\overline{K}}/\tau$ are $9A_2$. These are given by the images of $A[1-\tau](\overline{K})$ under $q$. Therefore, there are nine singularities of the same type by translation. It suffices to show that the singularity at $q(O)$ is of type $A_2$.

Let $\widehat{\cal{O}}_{A,O}$ be the completed local ring of $A$ at $O$. Since $\tau^2+\tau+1=0$ and $\tau$ is symplectic, we may choose $x,y\in \widehat{\cal{O}}_{A,O}$ such that $\widehat{\cal{O}}_{A,O}\cong K[\![x,y]\!]$ and
\[
    \tau:x\mapsto y\mapsto -x-y.
\]
Let $z=-x-y$. Then 
\[
    \widehat{\cal{O}}_{A,O}=\frac{K[\![x,y,z]\!]}{(x+y+z)},
\]
\[
    \tau:x\mapsto y\mapsto z.
\]
Let $u=xy+xz+yz,v=3xyz,w=xy^2+yz^2+x^2z$. Then the completed local ring of $A/\tau$ at $q(O)$ is
\[
    \widehat{\cal{O}}_{A,O}^{\tau=\text{id}}=\frac{K[\![u,v,w]\!]}{(u^3 + v^2 + vw + w^2)},
\]
which is indeed an $A_2$ singularity.

Next, we show that there is a $G_K$-equivariant bijection
\[
    S_K \cong A[1-\tau](\overline{K})\times \{\omega,\omega^2\}.
\]
The singularities on $A_{\overline{K}}/\tau$ are indexed $G_K$-equivariantly by $A[1-\tau](\overline{K})$. By translation, it suffices to show that the exceptional divisors of the $A_2$ singularity at $q(O)$ are indexed $G_K$-equivariantly by $\{\omega,\omega^2\}$.

The formal neighbourhood of $A_{\overline{K}}/\tau$ at $q(O)$ is $\Spec \widehat{\cal{O}}_{A,O}^{\tau=\text{id}}$. We blow it up at $(u,v,w)$, and focus on the affine chart $D(u)$ of the blowup. Let $v'=v/u, w'=w/u$. Then
\[
    D(u)=V(u^2+v'^2+v'w'+w'^2)
\]
with two exceptional divisors
\[
    V(u,v'^2+v'w'+w'^2)=V(u,v'-\omega w')\cup V(u,v'-\omega^2w'),
\]
which are indeed indexed $G_K$-equivariantly by $\{\omega,\omega^2\}$.

\subsubsection{Singularities and Galois action on the special fibre}\label{Singularities and Galois action on the special fibre order 3}

We now describe the singularities of $\cal{A}_{\overline{k}}/\tau$ and how $G_k$ acts on $S_k$.

\begin{rem}
    The argument only requires $\tau^2+\tau+1=0$, and that the base characteristic is $3$. In particular, this applies to $\cal{A}_{\overline{k}}/\tau^2$ in the order $6$, $p=3$ case.
\end{rem}

We first show that the singularities of $\cal{A}_{\overline{k}}/\tau$ are $3E_6^1$. These are given by the images of $\cal{A}[1-\tau](\overline{k})$ under $q_k$. Therefore, there are three singularities of the same type by translation. It suffices to show that the singularity at $q_k(O_k)$ is of type $E_6^1$.

To compute the $\tau$-action formally locally, let $\widehat{\cal{A}}_{\overline{k}}$ be the formal group of $\cal{A}_{\overline{k}}$. Let $X^*(\widehat{\cal{A}}_{\overline{k}})=\Hom(\widehat{\cal{A}}_{\overline{k}},\widehat{\bb{G}}_m)$ be the character group of $\widehat{\cal{A}}_{\overline{k}}$. By \cite[Lemma 4]{kummer_in_char_2_katsura}, $\widehat{\cal{A}}_{\overline{k}}\cong \widehat{\bb{G}}^{\oplus2}_m$, and $X^*(\widehat{\cal{A}}_{\overline{k}})$ is a free $\bb{Z}_3$-module of rank $2$. Then $\Aut(\widehat{\cal{A}}_{\overline{k}})\cong\Aut(X^*(\widehat{\cal{A}}_{\overline{k}}))\cong\GL_2(\bb{Z}_3)$.

We view $X^*(\widehat{\cal{A}}_{\overline{k}})$ as a free $\bb{Z}_3[\omega]$-module of rank $1$ where $\omega$ acts as $\tau$. We may therefore identify $X^*(\widehat{\cal{A}}_{\overline{k}})$ with $\bb{Z}_3[\omega]$ and choose $\{1,\omega\}$ as a basis of $X^*(\widehat{\cal{A}}_{\overline{k}})$ over $\bb{Z}_3$. Then $\tau$ corresponds to $\begin{psmallmatrix} 0 & -1 \\ 1 & -1 \end{psmallmatrix}\in \GL_2(\bb{Z}_3)$. Explicitly, the coordinate ring $\widehat{\cal{O}}_{\cal{A}_{\overline{k}},O_k}$ of $\widehat{\cal{A}}_{\overline{k}}$ is $\overline{k}[X, Y]^\wedge_{(X-1,Y-1)}$, the completion of $\overline{k}[X,Y]$ at $(X-1,Y-1)$, and
\[
    \tau:X\mapsto Y\mapsto X^{-1}Y^{-1}.
\]
Set $x=X-1,y=Y-1$. Then we have
\[
    \widehat{\cal{O}}_{\cal{A}_{\overline{k}},O_k}\cong \overline{k}[\![x,y]\!],
\]
\[
    \tau:x\mapsto y\mapsto \frac{1}{(x+1)(y+1)}-1.
\]
Let $z=\frac{1}{(x+1)(y+1)}-1$, so we have
\[
    \widehat{\cal{O}}_{\cal{A}_{\overline{k}},O_k}=\frac{\overline{k}[\![x,y,z]\!]}{((x+1)(y+1)(z+1)-1)},
\]
\[
    \tau:x\mapsto y\mapsto z.
\]
Let $u=x+y+z,v=xy+xz+yz, w=(x-y)(x-z)(y-z)$. Then the completed local ring of $\cal{A}_{\overline{k}}/\tau$ at $q_k(O_k)$ is
\[
    \widehat{\cal{O}}_{\cal{A}_{\overline{k}},O_k}^{\tau=\text{id}}=\frac{\overline{k}[\![u,v,w]\!]}{(w^2-u^4-u^3v-u^2v^2+v^3)},
\]
which is indeed an $E_6^1$ singularity.

Next, we show that there is a $G_k$-equivariant bijection
\[
    S_k\cong \cal{A}[1-\tau](\overline{k})\times D(\cal{A}[1-\tau]^\circ)(\overline{k})\times \underline{2}.
\]
The singularities on $\cal{A}_{\overline{k}}/\tau$ are indexed $G_k$-equivariantly by $\cal{A}[1-\tau](\overline{k})$. By translation, it suffices to show that the exceptional divisors of the $E_6^1$ singularity at $q_k(O_k)$ are indexed $G_k$-equivariantly by $D(\cal{A}[1-\tau]^\circ)(\overline{k})\times \underline{2}$.

Let $\Aut(E_6)$ be the automorphism group of the exceptional divisors of the $E_6^1$ singularity at $q_k(O_k)$. Explicitly, $\Aut(E_6)\cong \bb{Z}/2$ where the non-trivial element exchanges the two long arms of the $E_6$ graph. Denote the $G_k$-action on these exceptional divisors as
\[
    \chi_1:G_k\to \Aut(E_6).
\]

Since $D(\cal{A}[1-\tau]^\circ)(\overline{k})$ has order $3$, $\Aut(D(\cal{A}[1-\tau]^\circ)(\overline{k}))\cong \bb{Z}/2$. Denote by 
\[
    \chi_2:G_k\to \Aut(D(\cal{A}[1-\tau]^\circ)(\overline{k}))
\]
the action of $G_k$ on $D(\cal{A}[1-\tau]^\circ)(\overline{k})$. To prove that the exceptional divisors of the $E_6^1$ singularity at $q_k(O_k)$ are indexed $G_k$-equivariantly by $D(\cal{A}[1-\tau]^\circ)(\overline{k})\times \underline{2}$, it suffices to show that $\chi_1=\chi_2$.

Let
\[
    \chi:G_k\to \Aut(\widehat{\cal{A}}_{\overline{k}})\cong \Aut(X^*(\widehat{\cal{A}}_{\overline{k}}))\cong \GL_2(\bb{Z}_3)
\]
be the action of $G_k$ on $\widehat{\cal{A}}_{\overline{k}}$. Since the action of $G_k$ on $\widehat{\cal{A}}_{\overline{k}}$ commutes with $\tau$, the image of $\chi$ lies in the centralizer of $\tau$. After we view $X^*(\widehat{\cal{A}}_{\overline{k}})$ as a free $\bb{Z}_3[\omega]$-module of rank $1$, the centralizer of $\tau$ is $\Aut_{\bb{Z}_3[\omega]}(X^*(\widehat{\cal{A}}_{\overline{k}}))=\bb{Z}_3[\omega]^\times$. Thus, we have $\chi:G_k\to \bb{Z}_3[\omega]^\times$.

The actions of $G_k$ on the exceptional divisors of the $E_6^1$ singularity at $q_k(O_k)$ and on $D(\cal{A}[1-\tau]^\circ)(\overline{k})=\cok(X^*(\widehat{\cal{A}}_{\overline{k}})\xrightarrow{1-\tau}X^*(\widehat{\cal{A}}_{\overline{k}}))$ can be recovered from the action of $G_k$ on $\widehat{\cal{A}}_{\overline{k}}$, i.e. we have the following commutative diagram:
\begin{equation}\label{first c.d.}
    \xymatrix{ & \Aut(E_6) \\
    G_k \ar[r]^\chi \ar[ur]^{\chi_1} \ar[dr]_-{\chi_2} & \bb{Z}_3[\omega]^\times \ar[u]_{f_1} \ar[d]^{f_2} \\
    & \Aut(D(\cal{A}[1-\tau]^\circ)(\overline{k}))
    }
\end{equation}

Let $\mathfrak{m}\subset \bb{Z}_3[\omega]$ be the maximal ideal. Then $1+\mathfrak{m}\subset \bb{Z}_3[\omega]^\times$ is a pro-$3$ subgroup. Since $\Aut(E_6)$ and $\Aut(D(\cal{A}[1-\tau]^\circ)(\overline{k}))$ have order $2$, it follows that $1+\mathfrak{m}$ is in the kernel of both $f_1$ and $f_2$. Therefore, $f_1$ and $f_2$ factor through $\bb{Z}_3[\omega]^\times/(1+\mathfrak{m})\cong \bb{F}_3^\times$, i.e. the following diagram commutes
\begin{equation}\label{final c.d.}
    \xymatrix{ & & \Aut(E_6) \\
    G_k \ar[r]^\chi & \bb{Z}_3[\omega]^\times \ar[ur]^{f_1} \ar[dr]_-{f_2} \ar[r] & \bb{F}_3^\times \ar[u]_{g_1} \ar[d]^{g_2} \\
    & & \Aut(D(\cal{A}[1-\tau]^\circ)(\overline{k}))
    }
\end{equation}
To prove that $\chi_1=\chi_2$, it remains to show $g_1=g_2$. Since $-1\in \bb{Z}_3[\omega]^\times$ maps to the non-trivial element in $\bb{F}_3^\times$, it suffices to prove $f_1(-1)=f_2(-1)$.

This can be checked explicitly since $-1\in \bb{Z}_3[\omega]^\times$ corresponds to the inversion automorphism $[-1]$ of $\cal{A}_{\overline{k}}$. The inversion automorphism $[-1]$ is non-trivial on $D(\cal{A}[1-\tau]^\circ)(\overline{k})$. Therefore, we need to check that $f_1(-1)$ is non-trivial, i.e. $[-1]$ exchanges the two long arms of the $E_6$ graph.

On the coordinate ring $\widehat{\cal{O}}_{\cal{A}_{\overline{k}},O_k}\cong \overline{k}[X, Y]^\wedge_{(X-1,Y-1)}$ of $\widehat{\cal{A}}_{\overline{k}}$, we have
\[
    [-1]:X\mapsto X^{-1}, Y\mapsto Y^{-1}.
\]
The formal neighbourhood of $\cal{A}_{\overline{k}}/\tau$ at $q_k(O_k)$ is $\Spec \widehat{\cal{O}}_{\cal{A}_{\overline{k}},O_k}^{\tau=\text{id}}$. We may blow it up at $(u,v,w)$, and focus on the affine chart $D(u)$ of the blowup. Let $v'=v/u, w'=w/u$. Then
\[
    D(u)=V(u^2 + u^2v' + u^2v'^2 - uv'^3 - w'^2),
\]
which is an $A_5$ singularity. The exceptional divisor is given by $V(u, w'^2)$, which corresponds to the vertex at the end of the shortest arm of the $E_6$ graph.

We may further blow up this affine chart $D(u)$ at $(u,v',w')$, and focus on the affine chart $D(v')$ of the blowup. Let $u'=u/v', w''=w'/v'$. Then
\[
    D(v')=V(u'^2+u'^2v'+u'^2v'^2-u'v'^2-w''^2),
\]
which is an $A_3$ singularity. The exceptional divisors are given by
\[
    V(v', u'^2-w''^2)=V(v',u'-w'')\cup V(v',u'+w''),
\]
which correspond to the vertices at the end of the two long arms of $E_6$. By explicit computation, we check that $[-1]:X\mapsto X^{-1}, Y\mapsto Y^{-1}$ exchanges these two exceptional divisors, and therefore exchanges the two long arms of the $E_6$ graph. This finishes the computation of the $G_k$-action on $S_k$.

\subsection{The order $4$, $p=2$ case}

We now prove Proposition \ref{computation summary}(3).

\subsubsection{Singularities and Galois action on the generic fibre}\label{Singularities and Galois action on the generic fibre order 4}

The singularities on $A_{\overline{K}}/\tau$ are given by the images of $A[1-\tau^2](\overline{K})=A[2](\overline{K})$ under $q$. Consider the following factorization of $q$:
\[
    q:A\to A/\tau^2\overset{q'}{\to} A/\tau.
\]
Since $\tau^2=-1$, there are sixteen $A_1$ singularities on $A_{\overline{K}}/\tau^2$ by Section \ref{Singularities and Galois action on the generic fibre order 2}. Four of them in $A[1-\tau](\overline{K})$ are fixed by $\tau$, which give four singularities on $A_{\overline{K}}/\tau$.

The remaining twelve singularities on $A_{\overline{K}}/\tau^2$ split into six $\tau$-orbits, which give six singularities on $A_{\overline{K}}/\tau$. Note that $q'$ is finite \'etale away from the fixed points of $\tau$. Thus, these six singularities have type $A_1$. Both these singularities and their exceptional divisors are indexed $G_K$-equivariantly by
\[
    (A[2](\overline{K})\setminus A[1-\tau](\overline{K}))/\tau.
\]

The singularities arising from $A[1-\tau](\overline{K})$ have the same type by translation. Therefore, we focus on the singularity at $q(O)$. Let $\widehat{\cal{O}}_{A,O}$ be the completed local ring of $A$ at $O$. Since $\tau^2+1=0$ and $\tau$ is symplectic, we may choose $x,y\in \widehat{\cal{O}}_{A,O}$ such that $\widehat{\cal{O}}_{A,O}\cong K[\![x,y]\!]$ and
\[
    \tau:x\mapsto y\mapsto -x\mapsto -y.
\]
Let $u=x^2+y^2,v=x^2y^2,w=x^3y-xy^3$. Then the completed local ring of $A/\tau$ at $q(O)$ is
\[
    \widehat{\cal{O}}_{A,O}^{\tau=\text{id}}=\frac{K[\![u,v,w]\!]}{(u^2v-4v^2-w^2)},
\]
which is an $A_3$ singularity.

The four $A_3$ singularities are indexed $G_K$-equivariantly by $A[1-\tau](\overline{K})$. By translation, it remains to compute how $G_K$ acts on the exceptional divisors of the $A_3$ singularity at $q(O)$. We blow $\widehat{\cal{O}}_{A,O}^{\tau=\text{id}}$ up at $(u,v,w)$, and focus on the affine chart $D(u)$ of the blowup. Let $v'=v/u, w'=w/u$. Then
\[
    D(u)=V(uv'-4v'^2-w'^2)
\]
with exceptional divisors
\[
    V(u,4v'^2+w'^2)=V(u,2v'+iw')\cup V(u,2v'-iw'),
\]
corresponding to the two end vertices of the $A_3$ graph. Therefore, the exceptional divisors of the $A_3$ singularity are indexed $G_K$-equivariantly by $\{1,\pm i\}$. This contributes the factor $A[1-\tau](\overline{K})\times \{1,\pm i\}$ to $S_K$.

\subsubsection{Singularities and Galois action on the special fibre}\label{Singularities and Galois action on the special fibre order 4}

The singularities of $\cal{A}_{\overline{k}}/\tau$ are given by the images of $\cal{A}[1-\tau^2](\overline{k})=\cal{A}[2](\overline{k})$ under $q_k$. Similar to the reasoning in Section \ref{Singularities and Galois action on the generic fibre order 4}, by translation, the two points in $\cal{A}[1-\tau](\overline{k})$ give two singularities of the same type on $\cal{A}_{\overline{k}}/\tau$, which we determine later. The remaining two points in $\cal{A}[2](\overline{k})\setminus \cal{A}[1-\tau](\overline{k})$ form a single $\tau$-orbit and give one singularity on $\cal{A}_{\overline{k}}/\tau$, which has type $D_4^1$ by Section \ref{Singularities and Galois action on the special fibre order 2}.

We first compute how $G_k$ acts on the exceptional divisors of this $D_4^1$ singularity. Let $Y\to \cal{A}_{\overline{k}}/\tau^2$ and $X\to \cal{A}_{\overline{k}}/\tau$ be minimal desingularizations. Note that $\tau$ extends to $Y$ and $X\to \cal{A}_{\overline{k}}/\tau$ factors through $Y/\tau$. By Section \ref{Singularities and Galois action on the special fibre order 2}, the exceptional divisors of each $D_4^1$ singularity on  $\cal{A}_{\overline{k}}/\tau^2$ are indexed by $\cal{A}[2]^\circ(\overline{K})$. In particular, the exceptional divisors of the two $D_4^1$ singularities arising from $\cal{A}[2](\overline{k})\setminus \cal{A}[1-\tau](\overline{k})$ are indexed by
\[
    (\cal{A}[2](\overline{k})\setminus \cal{A}[1-\tau](\overline{k}))\times \cal{A}[2]^\circ(\overline{K}).
\]
Their image in $Y/\tau$ forms a $D_4$ configuration, which corresponds to the exceptional divisors of $X\to \cal{A}_{\overline{k}}/\tau$ above the $D_4^1$ singularity. This contributes the factor $((\cal{A}[2](\overline{k})\setminus \cal{A}[1-\tau](\overline{k}))\times \cal{A}[2]^\circ(\overline{K}))/\tau$ to $S_k$.

It remains to compute the singularity at $q_k(O_k)$ and how $G_k$ permutes its exceptional divisors. This is similar to the formal local computation in Section \ref{Singularities and Galois action on the special fibre order 3}. Let $\widehat{\cal{A}}_{\overline{k}}\cong \widehat{\bb{G}}^{\oplus2}_m$ be the formal group of $\cal{A}_{\overline{k}}$. Let $X^*(\widehat{\cal{A}}_{\overline{k}})=\Hom(\widehat{\cal{A}}_{\overline{k}},\widehat{\bb{G}}_m)$ be the character group of $\widehat{\cal{A}}_{\overline{k}}$, viewed as a free $\bb{Z}_2[i]$-module of rank $1$, where $i$ acts as $\tau$. We may therefore identify $X^*(\widehat{\cal{A}}_{\overline{k}})$ with $\bb{Z}_2[i]$ and choose $\{1,i\}$ as a basis of $X^*(\widehat{\cal{A}}_{\overline{k}})$ over $\bb{Z}_2$. Then $\tau$ corresponds to $\begin{psmallmatrix} 0 & -1 \\ 1 & 0 \end{psmallmatrix}\in \GL_2(\bb{Z}_2)$. Explicitly, the coordinate ring $\widehat{\cal{O}}_{\cal{A}_{\overline{k}},O_k}$ of $\widehat{\cal{A}}_{\overline{k}}$ is $\overline{k}[X, Y]^\wedge_{(X-1,Y-1)}$, the completion of $\overline{k}[X,Y]$ at $(X-1,Y-1)$, and
\[
    \tau:X\mapsto Y\mapsto X^{-1}\mapsto Y^{-1}.
\]
Set $x=X-1,y=Y-1$. Then we have
\[
    \widehat{\cal{O}}_{\cal{A}_{\overline{k}},O_k}\cong \overline{k}[\![x,y]\!],
\]
\[
    \tau:x\mapsto y\mapsto \frac{x}{x+1}\mapsto\frac{y}{y+1}.
\]
Let $z=\frac{x}{x+1}, w=\frac{y}{y+1}$, so we have
\[
    \widehat{\cal{O}}_{\cal{A}_{\overline{k}},O_k}=\frac{\overline{k}[\![x,y,z,w]\!]}{(xz+x+z,yw+y+w)},
\]
\[
    \tau:x\mapsto y\mapsto z\mapsto w.
\]
Let $u=x+y+z+w,v=xy+xw+yz+zw, r=xy^2+yz^2+zw^2+x^2w$. Then the completed local ring of $\cal{A}_{\overline{k}}/\tau$ at $q_k(O_k)$ is
\[
    \widehat{\cal{O}}_{\cal{A}_{\overline{k}},O_k}^{\tau=\text{id}}=\frac{\overline{k}[\![u,v,r]\!]}{(u^3v+uvr+v^3+r^2)},
\]
which is an $E_7^3$ singularity.

Since the automorphism group of an $E_7$ graph is trivial and the two $E_7^3$ singularities on $\cal{A}_{\overline{k}}/\tau$ are indexed $G_k$-equivariantly by $\cal{A}[1-\tau](\overline{k})\cong \bb{Z}/2$, it follows that $G_k$ acts trivially on their exceptional divisors, which contributes the factor $\underline{14}$ to $S_k$.

\subsection{The order $6$, $p=2$ case}

We now prove Proposition \ref{computation summary}(4).

\subsubsection{Singularities and Galois action on the generic fibre}\label{Singularities and Galois action on the generic fibre order 6, p=2}

The singularities of $A_{\overline{K}}/\tau$ are given by the images of $A[1-\tau^2](\overline{K})$ and $A[1-\tau^3](\overline{K})=A[2](\overline{K})$ under $q$. The origin $O$ is the only fixed point of $\tau$, giving one singularity on $A_{\overline{K}}/\tau$ whose type is determined later. The points in $A[1-\tau^2](\overline{K})\setminus 0$ split into four $\tau$-orbits and give four singularities on $A_{\overline{K}}/\tau$. By Section \ref{Singularities and Galois action on the generic fibre order 3}, they all have type $A_2$ and the exceptional divisors of each singularity are indexed $G_K$-equivariantly by $\{\omega,\omega^2\}$. This contributes the factor $(A[1-\tau^2](\overline{K})\setminus 0)/\tau\times \{\omega,\omega^2\}$ to $S_K$.

Next, the points in $A[2](\overline{K})\setminus 0$ split into five $\tau$-orbits, which give five singularities on $A_{\overline{K}}/\tau$. They all have type $A_1$ by Section \ref{Singularities and Galois action on the generic fibre order 2}. This contributes the factor $(A[2](\overline{K})\setminus 0)/\tau$ to $S_K$.

It remains to compute the type of singularities at $q(O)$ and how $G_K$ acts on its exceptional divisors. Let $\widehat{\cal{O}}_{A,O}$ be the completed local ring of $A$ at the origin. Since $\tau^2-\tau+1=0$ and $\tau$ is symplectic, we may choose $x,y\in \widehat{\cal{O}}_{A,O}$ such that $\widehat{\cal{O}}_{A,O}\cong K[\![x,y]\!]$ and
\[
    \tau:x\mapsto y\mapsto -x+y\mapsto -x\mapsto -y\mapsto x-y.
\]
Let $z=-x+y, w=-x-z$. Then 
\[
    \widehat{\cal{O}}_{A,O}=\frac{K[\![x,z,w]\!]}{(x+z+w)},
\]
\[
    \tau:x\mapsto -w\mapsto z\mapsto -x\mapsto w\mapsto -z.
\]
Let $u=xz+xw+zw, v=3xzw/2,r=(x-z)(z-w)(x-w)/2$. Then
\[
    \widehat{\cal{O}}_{A,O}^{\tau^2=\text{id}}=\frac{K[\![u,v,r]\!]}{(u^3+3v^2+r^2)},
\]
\[
    \tau:u\mapsto u, v\mapsto -v, r\mapsto -r.
\]
Let $a=v^2,b=vr$. Then the completed local ring of $A/\tau$ at $q(O)$ is
\[
    \widehat{\cal{O}}_{A,O}^{\tau=\text{id}}=\frac{K[\![u,a,b]\!]}{(u^3a + 3a^2 + b^2)},
\]
which is an $A_5$ singularity.

To compute how $G_K$ acts on its exceptional divisors, we blow $\widehat{\cal{O}}_{A,O}^{\tau=\text{id}}$ up at $(u,a,b)$, and focus on the affine chart $D(u)$ of the blowup. Let $a'=a/u, b'=b/u$. Then
\[
    D(u)=V(u^2a'+3a'^2+b'^2)
\]
with exceptional divisors
\[
    V(u,3a'^2+b'^2)=V(u,\sqrt{-3}a'+b')\cup V(u,-\sqrt{-3}a'+b'),
\]
corresponding to the two ends of the $A_5$ graph. This shows that the exceptional divisors of the $A_5$ singularity are indexed $G_K$-equivariantly by $\{1,\pm \omega,\pm\omega^2\}$.

\subsubsection{Singularities and Galois action on the special fibre}\label{Singularities and Galois action on the special fibre order 6, p=2}

The singularities of $\cal{A}_{\overline{k}}/\tau$ are given by the images of $\cal{A}[1-\tau^2](\overline{k})$ and $\cal{A}[1-\tau^3](\overline{k})=\cal{A}[2](\overline{k})$ on $\cal{A}_{\overline{k}}$ under $q_k$. The origin $O_k$ is the only fixed point of $\tau$, giving one singularity on $\cal{A}_{\overline{k}}/\tau$ whose type is determined later. Since $p=2$ and $\tau^2$ has order $3$, the points in $\cal{A}[1-\tau^2](\overline{k})\setminus 0$ give four $A_2$ singularities on $\cal{A}_{\overline{k}}/\tau$, and each singularity has exceptional divisors indexed $G_k$-equivariantly by $\{\omega,\omega^2\}$ as in Section \ref{Singularities and Galois action on the generic fibre order 6, p=2}. This contributes the factor
\[
    (\cal{A}[1-\tau^2](\overline{k})\setminus 0)\times \{\omega,\omega^2\}\cong (A[1-\tau^2](\overline{K})\setminus 0)\times \{\omega,\omega^2\}
\]
to $S_k$.

Next, the points in $\cal{A}[2](\overline{k})\setminus 0$ form a single $\tau$-orbit, which gives one $D_4^1$ singularity on $\cal{A}_{\overline{k}}/\tau$ by Section \ref{Singularities and Galois action on the special fibre order 2}. To compute how $G_k$ acts on its exceptional divisors, we use the same reasoning as in Section \ref{Singularities and Galois action on the special fibre order 4}: the exceptional divisors of each $D_4^1$ singularity on $\cal{A}_{\overline{k}}/\tau^3$ are indexed by $\cal{A}[2]^\circ(\overline{K})$ by Section \ref{Singularities and Galois action on the special fibre order 2}. In particular, the exceptional divisors of the singularities arising from $\cal{A}[2](\overline{k})\setminus 0$ are indexed by
\[
    (\cal{A}[2](\overline{k})\setminus 0)\times \cal{A}[2]^\circ(\overline{K}).
\]
Their quotient by $\tau$ forms a $D_4$ configuration, which corresponds to the exceptional divisors of the $D_4^1$ singularity on $\cal{A}_{\overline{k}}/\tau$. This contributes the factor $((\cal{A}[2](\overline{k})\setminus 0)\times \cal{A}[2]^\circ(\overline{K}))/\tau$ to $S_k$.

It remains to compute the singularity at $q_k(O_k)$ and how $G_k$ acts on its exceptional divisors. Let $\widehat{\cal{A}}_{\overline{k}}\cong \widehat{\bb{G}}^{\oplus2}_m$ be the formal group of $\cal{A}_{\overline{k}}$. Let $X^*(\widehat{\cal{A}}_{\overline{k}})=\Hom(\widehat{\cal{A}}_{\overline{k}},\widehat{\bb{G}}_m)$ be the character group of $\widehat{\cal{A}}_{\overline{k}}$, viewed as a free $\bb{Z}_2[\omega]$-module of rank $1$ where $-\omega$ acts as $\tau$. We may therefore identify $X^*(\widehat{\cal{A}}_{\overline{k}})$ with $\bb{Z}_2[\omega]$ and choose $\{1,-\omega\}$ as a basis of $X^*(\widehat{\cal{A}}_{\overline{k}})$ over $\bb{Z}_2$. Then $\tau$ corresponds to $\begin{psmallmatrix} 0 & -1 \\ 1 & 1 \end{psmallmatrix}\in \GL_2(\bb{Z}_2)$. Explicitly, the coordinate ring $\widehat{\cal{O}}_{\cal{A}_{\overline{k}},O_k}$ of $\widehat{\cal{A}}_{\overline{k}}$ is $\overline{k}[X, Y]^\wedge_{(X-1,Y-1)}$, the completion of $\overline{k}[X,Y]$ at $(X-1,Y-1)$, and
\[
    \tau:X\mapsto Y\mapsto X^{-1}Y\mapsto X^{-1}\mapsto Y^{-1}\mapsto XY^{-1}.
\]
Set $x=X-1,y=Y-1$. Then we have
\[
    \widehat{\cal{O}}_{\cal{A}_{\overline{k}},O_k}\cong \overline{k}[\![x,y]\!],
\]
\[
    \tau:x\mapsto y\mapsto \frac{x+y}{x+1}\mapsto\frac{x}{x+1} \mapsto\frac{y}{y+1}\mapsto\frac{x+y}{y+1}.
\]
Let $r=\frac{x^2}{x+1}, s=\frac{y^2}{y+1}, t=\frac{x^2+y^2}{(x+1)(y+1)}$, so we have
\[
    \widehat{\cal{O}}_{\cal{A}_{\overline{k}},O_k}^{\tau^3=\text{id}}=\frac{\overline{k}[\![r,s,t]\!]}{((r+s+t)^2+rst)},
\]
\[
    \tau:r\mapsto s\mapsto t.
\]
Let $u=rs+rt+st, v=rs^2+st^2+r^2t, w=r+s+t$. Then the completed local ring of $\cal{A}_{\overline{k}}/\tau$ at $q_k(O_k)$ is
\[
    \widehat{\cal{O}}_{\cal{A}_{\overline{k}},O_k}^{\tau=\text{id}}=\frac{\overline{k}[\![u,v,w]\!]}{(u^3+uvw+v^2+vw^2+w^5+w^4)},
\]
which is an $E_6^1$ singularity.

It remains to compute the $G_k$-action on the exceptional divisors of this $E_6^1$ singularity. Let $X\to \cal{A}_{\overline{k}}/\tau$ and $Y\to \cal{A}_{\overline{k}}/\tau^2$ be minimal desingularizations. Then $\tau$ extends to $Y$ and $X\to \cal{A}_{\overline{k}}/\tau$ factors through $Y/\tau$.

We claim that $\tau$ fixes the two exceptional divisors on $Y$ above the $A_2$ singularity arising from $O_k$, and their image in $Y/\tau$ consists of two rational curves whose intersection is a $D_4^1$ singularity: this can be verified by a formal local calculation similar to that above. After resolving the $D_4^1$ singularity, we obtain four more exceptional divisors. Together with the two rational curves, they form an $E_6$ configuration on $X$, which are exactly the exceptional divisors of the $E_6^1$ singularity $q_k(O_k)$. In particular, the two exceptional divisors on $Y$ above the $A_2$ singularity arising from $O_k$ correspond to the vertices at the ends of the two long arms of the $E_6$ graph. By Section \ref{Singularities and Galois action on the generic fibre order 3}, they are indexed by $\{\omega,\omega^2\}$. This proves that the exceptional divisors of the $E_6^1$ singularity are indexed by $\{1,\omega,\omega^2\}\times \underline{2}$.

\subsection{The order $6$, $p=3$ case}

We now prove Proposition \ref{computation summary}(5). The singularities on $A_{\overline{K}}/\tau$ and the $G_K$-action on $S_K$ are the same as in Section \ref{Singularities and Galois action on the generic fibre order 6, p=2}.

\subsubsection{Singularities and Galois action on the special fibre}\label{Singularities and Galois action on the special fibre order 6, p=3}

The singularities of $\cal{A}_{\overline{k}}/\tau$ are given by the images of $\cal{A}[1-\tau^2](\overline{k})$ and $\cal{A}[1-\tau^3](\overline{k})=\cal{A}[2](\overline{k})$ on $\cal{A}_{\overline{k}}$ under $q_k$. The origin $O_k$ is the only fixed point of $\tau$, giving one singularity on $\cal{A}_{\overline{k}}/\tau$ whose type is determined later. The points in $\cal{A}[2](\overline{k})\setminus 0$ split into five $\tau$-orbits, which give five $A_1$ singularities on $\cal{A}_{\overline{k}}/\tau$ by Section \ref{Singularities and Galois action on the generic fibre order 2}. These $A_1$ singularities, as well as their exceptional divisors, are indexed by
\[
    (\cal{A}[2](\overline{k})\setminus 0)/\tau\cong (A[2](\overline{K})\setminus 0)/\tau.
\]

Also, the points in $\cal{A}[1-\tau^2](\overline{k})\setminus 0$ form a single $\tau$-orbit, which gives one $E_6^1$ singularity on $\cal{A}_{\overline{k}}/\tau$ by Section \ref{Singularities and Galois action on the special fibre order 3}. To compute how $G_k$ acts on its exceptional divisors, recall that the exceptional divisors of each $E_6^1$ singularity on $\cal{A}_{\overline{k}}/\tau^2$ are indexed by $D(\cal{A}[1-\tau^2]^\circ)(\overline{k})\times \underline{2}$. In particular, the exceptional divisors of the singularities arising from $\cal{A}[1-\tau^2](\overline{k})\setminus 0$ are indexed by
\[
    (\cal{A}[1-\tau^2](\overline{k})\setminus 0)\times D(\cal{A}[1-\tau^2]^\circ)(\overline{k})\times \underline{2}.
\]
Their quotient by $\tau$ forms an $E_6$ configuration, which corresponds to the exceptional divisors of the $E_6^1$ singularity on $\cal{A}_{\overline{k}}/\tau$. This contributes the factor $((\cal{A}[1-\tau^2](\overline{k})\setminus 0)\times D(\cal{A}[1-\tau^2]^\circ)(\overline{k}))/\tau\times \underline{2}$ to $S_k$.

It remains to compute the singularity at $q_k(O_k)$ and how $G_k$ acts on its exceptional divisors. Let $\widehat{\cal{A}}_{\overline{k}}\cong \widehat{\bb{G}}^{\oplus2}_m$ be the formal group of $\cal{A}_{\overline{k}}$. Let $X^*(\widehat{\cal{A}}_{\overline{k}})=\Hom(\widehat{\cal{A}}_{\overline{k}},\widehat{\bb{G}}_m)$ be the character group of $\widehat{\cal{A}}_{\overline{k}}$, viewed as a free $\bb{Z}_3[\omega]$-module of rank $1$ where $-\omega$ acts as $\tau$. We may therefore identify $X^*(\widehat{\cal{A}}_{\overline{k}})$ with $\bb{Z}_3[\omega]$ and choose $\{1,\omega^2\}$ as a basis of $X^*(\widehat{\cal{A}}_{\overline{k}})$ over $\bb{Z}_3$. Then $\tau$ corresponds to $\begin{psmallmatrix} 1 & -1 \\ 1 & 0 \end{psmallmatrix}\in \GL_2(\bb{Z}_3)$. Explicitly, the coordinate ring $\widehat{\cal{O}}_{\cal{A}_{\overline{k}},O_k}$ of $\widehat{\cal{A}}_{\overline{k}}$ is $\overline{k}[X, Y]^\wedge_{(X-1,Y-1)}$, the completion of $\overline{k}[X,Y]$ at $(X-1,Y-1)$, and
\[
    \tau:X\mapsto XY\mapsto Y\mapsto X^{-1}\mapsto X^{-1}Y^{-1}\mapsto Y^{-1}.
\]
Set $x=X-1,y=Y-1$. Then we have
\[
    \widehat{\cal{O}}_{\cal{A}_{\overline{k}},O_k}\cong \overline{k}[\![x,y]\!],
\]
\[
    \tau:x\mapsto xy+x+y\mapsto y\mapsto \frac{-x}{x+1}\mapsto \frac{-xy-x-y}{(x+1)(y+1)}\mapsto \frac{-y}{y+1}.
\]
Let $z=\frac{-xy-x-y}{(x+1)(y+1)}$, so we have
\[
    \widehat{\cal{O}}_{\cal{A}_{\overline{k}},O_k}=\frac{\overline{k}[\![x,y,z]\!]}{((x+1)(y+1)(z+1)-1)},
\]
\[
    \tau:x\mapsto \frac{-z}{z+1}\mapsto  y \mapsto \frac{-x}{x+1} \mapsto z\mapsto \frac{-y}{y+1}.
\]
Let $u=x+y+z+xy+xz+yz,v=xy+xz+yz,w=(x-y)(x-z)(y-z)$. Then we have
\[
    \widehat{\cal{O}}_{\cal{A}_{\overline{k}},O_k}^{\tau^2=\text{id}}=\frac{\overline{k}[\![u,v,w]\!]}{(u^4+u^2v^2+v^4-v^3-w^2)},
\]
\[
    \tau:u\mapsto -u, v\mapsto v,w\mapsto -w.
\]
Let $a=u^2,b=uw$. Then the completed local ring of $\cal{A}_{\overline{k}}/\tau$ at $q_k(O_k)$ is
\[
    \widehat{\cal{O}}_{\cal{A}_{\overline{k}},O_k}^{\tau=\text{id}}=\frac{\overline{k}[\![a,b,v]\!]}{(av^4-av^3+a^2v^2+a^3-b^2)},
\]
which is an $E_7^1$ singularity. In particular, $G_k$ acts trivially on its exceptional divisors, which contributes the factor $\underline{7}$ to $S_k$.

\section{Good reduction criterion}

In this section, we prove Theorem \ref{main theorem}. Let $A$ be an abelian surface over $K$ with ordinary good reduction. Let $\tau$ be an automorphism of $A$ of order $3$, $4$, or $6$, such that the minimal desingularization $X=\Kum(A,\tau)$ of $A/\tau$ is a K3 surface. Suppose we are in the wild case.

\subsection{The general situation}

We first show that $\Kum(A,\tau)$ has potential good reduction. We begin by computing the generic and special fibres of $\cal{A}/\tau$. Since taking quotient by a finite group commutes with flat base change, $A/\tau\cong (\cal{A}/\tau)_K$. In fact, we also have $\cal{A}_k/\tau\cong (\cal{A}/\tau)_k$.

\begin{lem}\cite[c.f.][Proposition 5.6]{non-supersingular-kummer}\label{commutativity isomorphism}
    The natural map $\cal{A}_k\to (\cal{A}/\tau)_k$ induces an isomorphism
    \begin{equation*}
        \cal{A}_k/\tau\cong (\cal{A}/\tau)_k.
    \end{equation*}
\end{lem}
\begin{proof}
    Since taking quotient commutes with flat base change, we may assume that $k$ is algebraically closed. We prove this lemma case by case.

    \begin{enumerate}
        \item Suppose $\tau$ has order $3$ and $p=3$. Let $\pi$ be the projection $\cal{A}\to \cal{A}/\tau$ and let $V=\cal{A}[1-\tau](k)$. The action of $\tau$ on $\cal{A}_k$ is free away from $V$. Then we have
        \[
            (\cal{A}_k\setminus V)/\tau\cong (\cal{A}/\tau)_k\setminus \pi(V).
        \]
        It remains to prove $\cal{A}_k/\tau\cong (\cal{A}/\tau)_k$ around points in $V$. Again, by flatness, we may replace $\cal{A}$ by the spectra of the completed local rings at points in $V$, which are isomorphic by translation. Therefore, we focus on $O_k$, the origin of $\cal{A}_k$. We want to prove the natural map
        \[
            \phi:\widehat{\cal{O}}_{\cal{A},O_k}^{\tau=\text{id}}\otimes_{\cal{O}_K}k\to \widehat{\cal{O}}_{\cal{A}_k,O_k}^{\tau=\text{id}}
        \]
        is an isomorphism.

        For the injectivity of $\phi$, we prove that $\widehat{\cal{O}}_{\cal{A},O_k}^{\tau=\text{id}}\to \widehat{\cal{O}}_{\cal{A},O_k}$ remains injective after base change to $k$. It suffices to show that
        \[
            \Tor_1^{\cal{O}_K}(\widehat{\cal{O}}_{\cal{A},O_k}/\widehat{\cal{O}}_{\cal{A},O_k}^{\tau=\text{id}},k)=0,
        \]
        which follows since $\widehat{\cal{O}}_{\cal{A},O_k}/\widehat{\cal{O}}_{\cal{A},O_k}^{\tau=\text{id}}$ is $m_K$-torsion free. For the surjectivity of $\phi$, by the computation in Section \ref{Singularities and Galois action on the special fibre order 3}, there exists $X\in \widehat{\cal{O}}_{\cal{A}_k,O_k}\cong k[\![U,V]\!]$ such that $\widehat{\cal{O}}_{\cal{A}_k,O_k}^{\tau=\text{id}}$ is generated by 
        \begin{align*}
            u & = X+\tau(X)+\tau^2(X),\\
            v & = X\tau(X)+X\tau^2(X)+\tau(X)\tau^2(X),\\
            w & = (X-\tau(X))(X-\tau^2(X))(\tau(X)-\tau^2(X)),
        \end{align*}
        all of which visibly lift to $\widehat{\cal{O}}_{\cal{A},O_k}^{\tau=\text{id}}\cong \cal{O}_K[\![U,V]\!]^{\tau=\text{id}}$. This shows that $\cal{A}_k/\tau\cong (\cal{A}/\tau)_k$ in the order $3$, $p=3$ case. Note that the proof only requires $p=3$ and $\tau^2+\tau+1=0$. In particular, we have $\cal{A}_k/\tau^2\cong (\cal{A}/\tau^2)_k$ in the order $6$, $p=3$ case.
        \item Suppose $\tau$ has order $4$ and $p=2$. Since $\tau^2=-1$, by \cite[Proposition 5.6]{non-supersingular-kummer}, we have $\cal{A}_k/\tau^2\cong(\cal{A}/\tau^2)_k$. To prove $\cal{A}_k/\tau\cong (\cal{A}/\tau)_k$, it suffices to show
        \[
            ((\cal{A}/\tau^2)/\tau)_k\cong (\cal{A}/\tau^2)_k/\tau,
        \]
        which holds away from the image of $\cal{A}[1-\tau](k)$ in $\cal{A}_k/\tau^2$. Therefore, it remains to prove $\cal{A}_k/\tau\cong (\cal{A}/\tau)_k$ around points in $\cal{A}[1-\tau](k)$.

        By flatness and translation, we may replace $\cal{A}$ by the spectrum of the completed local ring at $O_k$. Then the proof is completely analogous to the order $3$, $p=3$ case: the injectivity follows by the same argument, and the surjectivity follows by the computation in Section \ref{Singularities and Galois action on the special fibre order 4}.
        \item  Suppose $\tau$ has order $6$ and $p=2$. Since $\tau^3=-1$, we have $\cal{A}_k/\tau^3\cong(\cal{A}/\tau^3)_k$ by \cite[Proposition 5.6]{non-supersingular-kummer}. It remains to prove
        \[
            ((\cal{A}/\tau^3)/\tau)_k\cong (\cal{A}/\tau^3)_k/\tau,
        \]
        which follows since $\tau$ has order $3$ on $\cal{A}/\tau^3$ and $p=2$.
        \item The order $6$, $p=3$ case is completely analogous to the order $6$, $p=2$ case. We have
        \[
            \cal{A}_k/\tau=(\cal{A}_k/\tau^2)/\tau\cong (\cal{A}/\tau^2)_k/\tau\cong (\cal{A}/\tau)_k,
        \]
        by the proof of the order $3$, $p=3$ case and the tameness of $\tau$ on $\cal{A}/\tau^2$.
    \end{enumerate}
\end{proof}

\begin{thm}\label{potential good redn}
    Let $A$ be an abelian surface over $K$ with ordinary good reduction. Let $\tau$ be an automorphism of $A$ of order $3$, $4$, or $6$, such that the minimal desingularization $X=\Kum(A,\tau)$ of $A/\tau$ is a K3 surface. Suppose we are in the wild case. Then there exist:
    \begin{enumerate}
        \item a finite extension $L$ of $K$ with residue field $k_L$, a smooth algebraic space $\cal{X}$ over $\cal{O}_L$;
        \item a proper morphism $\phi:\cal{X}\to \cal{A}_{\cal{O}_L}/\tau$ such that $\phi_L$ and $\phi_{k_L}$ are respectively minimal resolutions of $A_L/\tau$ and $\cal{A}_{k_L}/\tau$.
    \end{enumerate}
    In particular, $\cal{X}$ is a smooth model of $X_L$, and $X$ has potential good reduction.
\end{thm}
\begin{proof}
    By construction, $\cal{A}/\tau$ is flat. Also, by Lemma \ref{commutativity isomorphism} and Proposition \ref{computation summary}, the generic fibre and special fibre of $\cal{A}/\tau$ are normal with rational double points. Then by \cite{artin_simultaneous_res}, after a finite extension $L/K$, we obtain the required morphism $\phi:\cal{X}\to \cal{A}_{\cal{O}_L}$. Since $X_L$ is the minimal desingularization of $A_L/\tau$, we obtain $\cal{X}_L\cong X_L$ and $\cal{X}$ is a smooth model of $X_L$.
\end{proof}

Let $\ell\neq p$ be a prime number. Denote by $S_K$ and $S_k$ respectively the sets of exceptional divisors of the minimal desingularizations of $A_{\overline{K}}/\tau$ and $\cal{A}_{\overline{k}}/\tau$, both viewed as $G_K$-sets as in Section \ref{computation section}. We will first give a criterion for the good reduction of $\Kum(A,\tau)$ that works for all orders of $\tau$.

\begin{prop}\label{general criterion}
    The following are equivalent:
    \begin{enumerate}
        \item the generalized Kummer surface $X=\Kum(A,\tau)$ has good reduction;
        \item there is a $G_K$-equivariant isomorphism $\bb{Q}_\ell^{S_K}\cong \bb{Q}_\ell^{S_k}$.
    \end{enumerate}
\end{prop}
\begin{proof}
    We may assume that $X$ has good reduction after a finite unramified extension. Indeed, this follows from either (1) or (2): this obviously holds if we have (1). On the other hand, suppose there is a $G_K$-equivariant isomorphism $\bb{Q}_\ell^{S_K}\cong \bb{Q}_\ell^{S_k}$. By Proposition \ref{cohomology computation}, we have
    \[
        H^2_{\et}(X_{\overline{K}},\bb{Q}_\ell)\cong H^2_{\et}(A_{\overline{K}},\bb{Q}_\ell)^{\tau=\text{id}}\oplus \bb{Q}_\ell^{S_K}(-1).
    \]
    Since $A$ has good reduction, $H^2_{\et}(A_{\overline{K}},\bb{Q}_\ell)^{\tau=\text{id}}$ is unramified. Also, $\bb{Q}_\ell^{S_K}\cong \bb{Q}_\ell^{S_k}$ implies that $\bb{Q}_\ell^{S_K}(-1)$ is unramified. Thus, $H^2_{\et}(X_{\overline{K}},\bb{Q}_\ell)$ is unramified. By Theorem \ref{potential good redn} and \cite{liedtke_2017_good_reduction_k3surfaces}, $X$ has good reduction after a finite unramified extension.

    Now assume $X$ has good reduction after a finite unramified extension. By Lemma \ref{find the caonical reduction}, $X^\circ$ is the minimal desingularization of $\cal{A}_k/\tau$. By Proposition \ref{cohomology computation}, there is a $G_K$-equivariant isomorphism
    \[
        H^2_{\et}(X^\circ_{\overline{k}},\bb{Q}_\ell)\cong H^2_{\et}(\cal{A}_{\overline{k}},\bb{Q}_\ell)^{\tau=\text{id}}\oplus \bb{Q}_\ell^{S_k}(-1).
    \]
    By Theorem \ref{comparison theorem}, $X$ has good reduction if and only if $H^2_{\et}(X_{\overline{K}},\bb{Q}_\ell)\cong H^2_{\et}(X^\circ_{\overline{k}},\bb{Q}_\ell)$ as $G_K$-representations. Then it suffices to show that $\bb{Q}_\ell^{S_K}\cong \bb{Q}_\ell^{S_k}$ if and only if $H^2_{\et}(X_{\overline{K}},\bb{Q}_\ell)\cong H^2_{\et}(X^\circ_{\overline{k}},\bb{Q}_\ell)$. Since $A$ has good reduction, $H^2_{\et}(A_{\overline{K}},\bb{Q}_\ell)^{\tau=\text{id}}\cong H^2_{\et}(\cal{A}_{\overline{k}},\bb{Q}_\ell)^{\tau=\text{id}}$. Then the ``only if'' implication is obvious, so we are left with the ``if'' implication.
    
    Note that $\bb{Q}_\ell^{S_K}$ and $\bb{Q}_\ell^{S_k}$ are semi-simple since the $G_K$-actions factor through a finite quotient. Suppose $H^2_{\et}(X_{\overline{K}},\bb{Q}_\ell)\cong H^2_{\et}(X^\circ_{\overline{k}},\bb{Q}_\ell)$. Tensoring with $\bb{Q}_\ell(1)$ and passing to the semi-simplification, we obtain
    \[
        H^2_{\et}(A_{\overline{K}},\bb{Q}_\ell(1))^{\tau=\text{id},ss}\oplus \bb{Q}_\ell^{S_K}\cong H^2_{\et}(\cal{A}_{\overline{k}},\bb{Q}_\ell(1))^{\tau=\text{id},ss}\oplus \bb{Q}_\ell^{S_k}.
    \]
    Since $A$ has good reduction, $H^2_{\et}(A_{\overline{K}},\bb{Q}_\ell(1))^{\tau=\text{id},ss}\cong H^2_{\et}(\cal{A}_{\overline{k}},\bb{Q}_\ell(1))^{\tau=\text{id},ss}$. After cancellation, we obtain $\bb{Q}_\ell^{S_K}\cong \bb{Q}_\ell^{S_k}$.
\end{proof}

Now we substitute our computation of $S_K$ and $S_k$ in Proposition \ref{computation summary} into Proposition \ref{general criterion} to prove Theorem \ref{main theorem}. We do this case by case.

\subsection{The order $3$, $p=3$ case}

We now prove Theorem \ref{main theorem}(1). By Proposition \ref{general criterion}, it suffices to prove that there is a $G_K$-equivariant isomorphism $\bb{Q}_\ell^{S_K}\cong \bb{Q}_\ell^{S_k}$ if and only if $\omega\in K$ and there is a $G_K$-equivariant splitting to the short exact sequence
\[
    0\to \cal{A}[1-\tau]^\circ(\overline{K})\to A[1-\tau](\overline{K})\to \cal{A}[1-\tau](\overline{k})\to 0.
\]

By Proposition \ref{computation summary}, we have $G_K$-equivariant bijections:
\begin{align*}
    S_K & \cong A[1-\tau](\overline{K}) \times \{\omega, \omega^2\},\\
    S_k & \cong \mathcal{A}[1-\tau](\overline{k})\times D(\cal{A}[1-\tau]^\circ)(\overline{k})\times \underline{2}.
\end{align*}
Let $V_1=\cal{A}[1-\tau]^\circ(\overline{K})$, $V_2=\cal{A}[1-\tau](\overline{k})$, $V=V_1\oplus V_2$, $W=A[1-\tau](\overline{K})$.

\begin{lem}\label{order 3 refinement}
    The following are equivalent:
    \begin{enumerate}
        \item there is a $G_K$-equivariant isomorphism $\bb{Q}_\ell^{S_K}\cong \bb{Q}_\ell^{S_k}$;
        \item $\omega\in K$ and there is a $G_K$-equivariant isomorphism $\bb{Q}_\ell^W\cong \bb{Q}_\ell^V$.
    \end{enumerate}
\end{lem}
\begin{proof}
    We may assume that $\omega\in K$. Indeed, this is part of (2) and it also follows if (1) holds: suppose there is a $G_K$-equivariant isomorphism $\bb{Q}_\ell^{S_K}\cong \bb{Q}_\ell^{S_k}$. Then for every $\sigma\in G_K$, the traces of $\sigma$ on $\bb{Q}_\ell^{S_K}$ and $ \bb{Q}_\ell^{S_k}$ are the same. In particular, the numbers of fixed points of $\sigma$ on $S_K$ and $S_k$ are the same. Suppose $\omega\notin K$. Then there exists $\sigma\in G_K$ such that $\sigma(\omega)=\omega^2$. Then $\sigma$ has no fixed points on $S_K$, while $\sigma$ has at least two fixed points on $S_k$. This is a contradiction.
    
    Now assume $\omega\in K$. Then by Corollary \ref{group scheme embedding special case}, $D(\cal{A}[1-\tau]^\circ)(\overline{k})\cong \cal{A}[1-\tau]^\circ(\overline{K})=V_1$. In particular, we have $G_K$-equivariant bijections $S_K\cong W\times \underline{2}$ and $S_k\cong V\times \underline{2}$.
    Therefore, as $G_K$-representations, $\bb{Q}_\ell^{S_K}\cong \bb{Q}_\ell^{S_k}$ if and only if $\bb{Q}_\ell^W\cong \bb{Q}_\ell^V$ by semi-simplicity.
\end{proof}

We are now ready to prove Theorem \ref{main theorem}(1).

\begin{proof}[Proof of Theorem \ref{main theorem}(1)]
    This proof is analogous to that of \cite[Theorem 4.6]{non-supersingular-kummer}. By Lemma \ref{order 3 refinement}, it remains to prove that $\bb{Q}_\ell^W\cong \bb{Q}_\ell^V$ as $G_K$-representations if and only if there is a $G_K$-equivariant splitting to the short exact sequence
    \begin{equation}\label{VW exact sequence}
        0\to V_1\to W\to V_2\to 0.
    \end{equation}

    The ``if'' implication is clear. For the ``only if'' implication, assume $\bb{Q}_\ell^W\cong \bb{Q}_\ell^V$. Let $P\subset \GL(W)$ be the subgroup that preserves $V_1$. Let $G_W$ and $G_V$ respectively be the images of $G_K$ in $\GL(W)$ and $\GL(V_1)\times \GL(V_2)$. The natural homomorphism $\pi:P\to \GL(V_1)\times \GL(V_2)$ induces a surjection $G_W\to G_V$. Also, $\bb{Q}_\ell^W\cong \bb{Q}_\ell^V$ implies that for all $\sigma\in G_K$, $\sigma$ is trivial on $W$ if and only if it is on $V$. This implies that $\lvert G_W\rvert=\lvert G_V\rvert$, and therefore $G_W\cong G_V$.

    We first choose a section (possibly not $G_K$-equivariant) of $W\to V_2$. This induces a section $\sigma$ of $\pi:P\to \GL(V_1)\times \GL(V_2)$. Denote by $R_u(P)$ the kernel of $\pi$. Then the following short exact sequence of groups splits:
    \begin{equation}\label{general linear group split}
        1\to R_u(P)\to P\to \GL(V_1)\times \GL(V_2)\to 1.
    \end{equation}
    To prove that \eqref{VW exact sequence} splits, it suffices to show that $G_W$ is conjugate to $\sigma(G_V)$ by an element in $R_u(P)$.

    Take the pullback of \eqref{general linear group split} via $G_V\subset \GL(V_1)\times\GL(V_2)$ to get the following split short exact sequence:
    \[
        1\to R_u(P)\to R_u(P)\rtimes G_V\to G_V\to 1,
    \]
    where $G_V$ acts on $R_u(P)$ by conjugation. Since $G_W\cong \pi(G_W)=G_V$, it follows that $G_W$ is a complement to $R_u(P)$ in $R_u(P)\rtimes G_V$. It suffices to show that $H^1(G_V,R_u(P))=0$: indeed, by \cite[Proposition 9.21]{rotman_homological_algebra_2009}, this implies that all complements to $R_u(P)$ in $R_u(P)\rtimes G_V$ are conjugate. In particular, $\sigma(G_V)$ and $G_W$ are conjugate by some element in $R_u(P)$.

    Since $V_1$ and $V_2$ both have dimension $1$ over $\bb{F}_3$, $\GL(V_1)\times \GL(V_2)\cong \bb{Z}/2\times\bb{Z}/2$. In particular, $G_V$ is a $2$-group. Also, elements in $R_u(P)$ are of the form
    \[
        \begin{pmatrix}
            1 & a\\
            0 & 1
        \end{pmatrix}
    \]
    for $a\in \bb{F}_3$. Thus, $R_u(P)$ has order $3$. Then $H^1(G_V,R_u(P))$ is annihilated by both $2$ and $3$, thus trivial.
\end{proof}

\subsection{The order $4$, $p=2$ case}

We now prove Theorem \ref{main theorem}(2). By Proposition \ref{general criterion}, it suffices to prove that there is a $G_K$-equivariant isomorphism $\bb{Q}_\ell^{S_K}\cong \bb{Q}_\ell^{S_k}$ if and only if $i\in K$, $A[1-\tau](\overline{K})$ is trivial, and there is a $G_K\times\langle\tau\rangle$-equivariant splitting to the short exact sequence
\begin{equation}\label{order 4 sequence}
    0\to \cal{A}[2]^\circ(\overline{K})\to A[2](\overline{K})\overset{\rho}{\to} \cal{A}[2](\overline{k})\to 0.
\end{equation}

By Proposition \ref{computation summary}, we have $G_K$-equivariant bijections:
\begin{align*}
    S_K & \cong A[1-\tau](\overline{K}) \times \{1, \pm i\} \sqcup (A[2]({\overline{K}})\setminus A[1-\tau](\overline{K}))/\tau,\\
    S_k & \cong \underline{14}\sqcup (\mathcal{A}[2]^\circ(\overline{K})\times(\cal{A}[2](\overline{k})\setminus\cal{A}[1-\tau](\overline{k})))/\tau.
\end{align*}
We may simplify the problem. Let
\begin{align*}
    S_K' & =(A[2](\overline{K})\setminus (A[1-\tau](\overline{K})+\mathcal{A}[2]^\circ(\overline{K})))/\tau,\\
    S_k' & =(\mathcal{A}[2]^\circ(\overline{K})\times(\cal{A}[2](\overline{k})\setminus\cal{A}[1-\tau](\overline{k})))/\tau.
\end{align*}

\begin{lem}\label{order 4 refinement}
    The following are equivalent:
    \begin{enumerate}
        \item there is a $G_K$-equivariant isomorphism $\bb{Q}_\ell^{S_K}\cong \bb{Q}_\ell^{S_k}$;
        \item $i\in K$, $A[1-\tau](\overline{K})$ is trivial, and there is a $G_K$-equivariant isomorphism $\bb{Q}_\ell^{S_K'}\cong \bb{Q}_\ell^{S_k'}$.
    \end{enumerate}
\end{lem}
\begin{proof}
    First, we may assume that $i\in K$ and $A[1-\tau](\overline{K})$ is trivial. Indeed, they are part of (2) and they also follow if (1) holds: suppose there is a $G_K$-equivariant isomorphism $\bb{Q}_\ell^{S_K}\cong \bb{Q}_\ell^{S_k}$. Then for every $\sigma\in G_K$, the numbers of fixed points of $\sigma$ in $S_K$ and $S_k$ are the same. Suppose $i\notin K$. Then there exists $\sigma\in G_K$ such that $\sigma(i)=-i$. Then $\sigma$ has at most ten fixed points on $S_K$, while at least fourteen fixed points on $S_k$, a contradiction. Next, suppose $A[1-\tau](\overline{K})$ is non-trivial. Then there exists $\sigma\in G_K$ such that $\sigma$ has either one or two fixed points on $A[1-\tau](\overline{K})$. Then $\sigma$ has at most twelve fixed points on $S_K$, while at least fourteen fixed points on $S_k$, again a contradiction.

    Now assume $i\in K$ and $A[1-\tau](\overline{K})$ is trivial. Then
    \[
        S_K\cong 12\sqcup (A[2]({\overline{K}})\setminus A[1-\tau](\overline{K}))/\tau.
    \]
    Since $S_k=\underline{14}\sqcup S_k'$, by semi-simplicity, it suffices to prove
    \[
        (A[2]({\overline{K}})\setminus A[1-\tau](\overline{K}))/\tau\cong  S_K'\sqcup \underline{2},
    \]
    i.e. $(A[1-\tau](\overline{K})+\cal{A}[2]^\circ(\overline{K}))\setminus A[1-\tau](\overline{K})$ consists of two $\tau$-orbits, each preserved by $G_K$.

    First, we claim that the inclusion $\cal{A}[1-\tau]^\circ(\overline{K})\subset A[1-\tau](\overline{K})\cap \cal{A}[2]^\circ(\overline{K})$ is an equality. Indeed, both $A[1-\tau](\overline{K})$ and $\cal{A}[2]^\circ(\overline{K})$ have order $4$, and $A[1-\tau](\overline{K})\cap \cal{A}[2]^\circ(\overline{K})$ has order $2$ or $4$. To prove the claim, it remains to show that $A[1-\tau](\overline{K})\neq \cal{A}[2]^\circ(\overline{K})$, which is clear.

    In particular, this claim shows that $A[1-\tau](\overline{K})+\cal{A}[2]^\circ(\overline{K})$ has order $8$. Therefore, $(A[1-\tau](\overline{K})+\cal{A}[2]^\circ(\overline{K}))\setminus A[1-\tau](\overline{K})$ has cardinality $4$ and consists of two $\tau$-orbits. Explicitly, one orbit is
    \[
        \cal{A}[2]^\circ(\overline{K})\setminus \cal{A}[1-\tau]^\circ(\overline{K}),
    \]
    which is preserved by $G_K$. The other is
    \[
        (A[1-\tau](\overline{K})+\mathcal{A}[2]^\circ(\overline{K}))\setminus (A[1-\tau](\overline{K})\cup \mathcal{A}[2]^\circ(\overline{K})),
    \]
    also preserved by $G_K$.
\end{proof}

\begin{proof}[Proof of Theorem \ref{main theorem}(2)]
    By Lemma \ref{order 4 refinement}, it suffices to prove that there is a $G_K$-equivariant isomorphism $\bb{Q}_{\ell}^{S_K'}\cong \bb{Q}_{\ell}^{S_k'}$ if and only if there is a $G_K\times\langle\tau\rangle$-equivariant splitting to \eqref{order 4 sequence}.

    We begin with the ``only if'' direction. First, we claim that $G_K$ fixes some element in $S_K'$. To see this, note that each of the two $\tau$-orbits in
    \[
        \cal{A}[1-\tau]^\circ(\overline{K})\times(\cal{A}[2](\overline{k})\setminus \cal{A}[1-\tau](\overline{k}))
    \]
    is preserved by $G_K$. This gives two elements in $S_k'$ fixed by $G_K$. Thus, the only non-trivial $G_K$ action on $S_k'$ exchanges the remaining two elements. Then since $\bb{Q}_{\ell}^{S_K'}\cong \bb{Q}_{\ell}^{S_k'}$, the $G_K$-action on $S_k'$ and therefore on $S_K'$ factors through $G_K/H$ where $H$ is a subgroup of $G_K$ of index $2$. Pick $\sigma\in G_K\setminus H$. To prove the claim, it suffices to prove that $\sigma$ fixes some element in $S_K'$, which is clear since $\sigma$ fixes at least two elements in $S_k'$.

    Given this claim, let $\{P,\tau(P)\}$ be a $\tau$-orbit in $A[2](\overline{K})\setminus(\cal{A}[2]^\circ(\overline{K})+A[1-\tau](\overline{K}))$ preserved by $G_K$. Let $W=P+\tau(P)\in A[1-\tau](\overline{K})$. Note that 
    \[
        \rho^{-1}(\cal{A}[1-\tau](\overline{k}))=\cal{A}[2]^\circ(\overline{K})+A[1-\tau](\overline{K}).
    \]
    In particular, $\rho(P)\notin \cal{A}[1-\tau](\overline{k})$ and $\rho(P)\neq \rho(\tau(P))$. Therefore, $\rho(P)$ and $\rho(\tau(P))$ span $\cal{A}[2](\overline{k})$ and $\rho$ induces $\{0,P,\tau(P),W\}\cong \cal{A}[2](\overline{k})$. This shows that \eqref{order 4 sequence} splits $G_K\times\langle\tau\rangle$-equivariantly.

    Conversely, suppose there is a $G_K\times\langle\tau\rangle$-equivariant isomorphism $A[2](\overline{K})\cong \cal{A}[2]^\circ(\overline{K})\times \cal{A}[2](\overline{k})$. Then we have the following $G_K\times\langle\tau\rangle$-equivariant isomorphisms:
    \begin{align*}
        & A[2](\overline{K})\setminus (A[1-\tau](\overline{K})+\mathcal{A}[2]^\circ(\overline{K}))\\
        \cong{} & (\cal{A}[2]^\circ(\overline{K})\times \cal{A}[2](\overline{k}))\setminus (\cal{A}[2]^\circ(\overline{K})\times \cal{A}[1-\tau](\overline{k}))\\
        ={} & \cal{A}[2]^\circ(\overline{K})\times (\cal{A}[2](\overline{k})\setminus \cal{A}[1-\tau](\overline{k}))
    \end{align*}
    Taking the $\tau$-orbits gives a $G_K$-equivariant bijection $S_K'\cong S_k'$.
\end{proof}

\subsection{The order $6$, $p=2$ case}

We now prove Theorem \ref{main theorem}(3). By Proposition \ref{general criterion}, it suffices to prove that there is a $G_K$-equivariant isomorphism $\bb{Q}_\ell^{S_K}\cong \bb{Q}_\ell^{S_k}$ if and only if there is a $G_K$-equivariant splitting to the short exact sequence
\begin{equation}\label{order 6 p=2 sequence}
    0\to \cal{A}[2]^\circ(\overline{K})\to A[2](\overline{K})\overset{\rho}{\to} \cal{A}[2](\overline{k})\to 0.
\end{equation}

By Proposition \ref{computation summary}, we have $G_K$-equivariant bijections:
\begin{align*}
    S_K & \cong \{1, \pm \omega, \pm \omega^2\} \sqcup (A[1-\tau^2](\overline{K})\setminus 0)/\tau \times \{\omega, \omega^2\} \sqcup (A[2]({\overline{K}})\setminus 0)/\tau,\\
    S_k & \cong \{1,\omega,\omega^2\}\times\underline{2}\sqcup (A[1-\tau^2](\overline{K})\setminus 0)/\tau\times \{\omega, \omega^2\}\sqcup (\cal{A}[2]^\circ(\overline{K})\times (\cal{A}[2](\overline{k})\setminus 0))/\tau.
\end{align*}
Let
\begin{align*}
    S_K' & = (A[2]({\overline{K}})\setminus \cal{A}[2]^\circ(\overline{K}))/\tau,\\
    S_k' & = (\cal{A}[2]^\circ(\overline{K})\times (\cal{A}[2](\overline{k})\setminus 0))/\tau.
\end{align*}

\begin{lem}\label{order 6 p=2 refinement}
    As $G_K$-representations, $\bb{Q}_\ell^{S_K}\cong \bb{Q}_\ell^{S_k}$ if and only if $\bb{Q}_\ell^{S_K'}\cong \bb{Q}_\ell^{S_k'}$.
\end{lem}
\begin{proof}
    Since $S_k\cong \{1, \omega, \omega^2\}\times\underline{2} \sqcup (A[1-\tau^2](\overline{K})\setminus 0)/\tau \times \{\omega, \omega^2\}\sqcup S_k'$, it suffices to show $S_K\cong \{1, \omega, \omega^2\}\times\underline{2} \sqcup (A[1-\tau^2](\overline{K})\setminus 0)/\tau \times \{\omega, \omega^2\}\sqcup S_K'$ by semi-simplicity. Note that the three elements in $\cal{A}[2]^\circ(\overline{K})\setminus 0$ form a single $\tau$-orbit preserved by $G_K$. Thus, we obtain
    \begin{align*}
        S_K & \cong \{1, \pm \omega, \pm \omega^2\} \sqcup (A[1-\tau^2](\overline{K})\setminus 0)/\tau \times \{\omega, \omega^2\} \sqcup (\underline{1}\sqcup (A[2]({\overline{K}})\setminus \cal{A}[2]^\circ(\overline{K}))/\tau)\\
        & \cong \{1, \omega, \omega^2\}\times\underline{2} \sqcup (A[1-\tau^2](\overline{K})\setminus 0)/\tau \times \{\omega, \omega^2\} \sqcup (A[2]({\overline{K}})\setminus \cal{A}[2]^\circ(\overline{K}))/\tau\\
        & = \{1, \omega, \omega^2\}\times\underline{2} \sqcup (A[1-\tau^2](\overline{K})\setminus 0)/\tau \times \{\omega, \omega^2\}\sqcup S_K'.
    \end{align*}
\end{proof}

\begin{proof}[Proof of Theorem \ref{main theorem}(3)]
    By Lemma \ref{order 6 p=2 refinement}, it suffices to show that $\bb{Q}_{\ell}^{S_K'}\cong \bb{Q}_{\ell}^{S_k'}$ as $G_K$-representations if and only if there is a $G_K$-equivariant splitting to \eqref{order 6 p=2 sequence}.

    First, we make the following observations:
    \begin{enumerate}
        \item Since the Galois action commutes with $\tau$, any non-trivial Galois action on $\cal{A}[2](\overline{k})$ fixes $0$ and permutes the $\tau$-orbit $\cal{A}[2](\overline{k})\setminus 0$ cyclically.
        \item Similarly, the statement in (1) holds if we replace $\cal{A}[2](\overline{k})$ by $\cal{A}[2]^\circ(\overline{K})$, and $\cal{A}[2](\overline{k})\setminus 0$ by $\cal{A}[2]^\circ(\overline{K})\setminus 0$.
        \item Using (1) and (2), we see that any non-trivial $G_K$-action on $S_k'$ fixes the $\tau$-orbit $0\times (\cal{A}[2](\overline{k})\setminus 0)$ and permutes the remaining three $\tau$-orbits cyclically.
    \end{enumerate}

    Now we prove the ``only if'' implication. Using the $G_K$-equivariant isomorphism $\bb{Q}_{\ell}^{S_K'}\cong \bb{Q}_{\ell}^{S_k'}$ and observation (3), we see that $G_K$ fixes some element in $S_K'$. In particular, there exists a $\tau$-orbit $\{P,\tau(P),\tau^2(P)\}$ preserved by $G_K$ in $A[2](\overline{K})\setminus \cal{A}[2]^\circ(\overline{K})$. Note that $\tau^2(P)=\tau(P)+P$. Then $\{0, P,\tau(P),\tau^2(P)\}\subset A[2](\overline{K})$ is a subgroup that maps isomorphically onto $\cal{A}[2](\overline{k})$ via $\rho$ and provides a $G_K$-equivariant splitting to \eqref{order 6 p=2 sequence}.

    Conversely, suppose there is a $G_K$-equivariant isomorphism $A[2](\overline{K})\cong \cal{A}[2]^\circ(\overline{K})\times \cal{A}[2](\overline{k})$. Fix an arbitrary $\sigma\in G_K$. It suffices to show that $\sigma$ on $\bb{Q}_{\ell}^{S_K'}$ and $ \bb{Q}_{\ell}^{S_k'}$ has the same trace, i.e. $\sigma$ on $S_K'$ and $S_k'$ fixes the same number of elements. Equivalently, we prove that there is a $\sigma$-equivariant isomorphism $S_K'\cong S_k'$.
    
    By observations (1) and (2), $\sigma^3$ is trivial on $\cal{A}[2]^\circ(\overline{K})$ and $\cal{A}[2](\overline{k})$, therefore trivial on $A[2](\overline{K})$ and $S_K'$. In particular, $\sigma$ has at least one fixed $\tau$-orbit in $A[2]({\overline{K}})\setminus \cal{A}[2]^\circ(\overline{K})$. Let this $\tau$-orbit be $\{P,\tau(P),\tau^2(P)\}$. Similar to the reasoning in the ``only if'' implication, one checks that $\{0, P,\tau(P),\tau^2(P)\}$ provides a $\langle\sigma,\tau\rangle$-equivariant splitting to \eqref{order 6 p=2 sequence}. In particular, there is a $\langle\sigma,\tau\rangle$-equivariant isomorphism
    \[
         A[2](\overline{K})\setminus \cal{A}[2]^\circ(\overline{K})
         \cong \cal{A}[2]^\circ(\overline{K})\times (\cal{A}[2](\overline{k})\setminus 0).
    \]
    After taking the $\tau$-orbits, we obtain a $\sigma$-equivariant isomorphism $S_K'\cong S_k'$.
\end{proof}

\subsection{The order $6$, $p=3$ case}

We now prove Theorem \ref{main theorem}(4). By Proposition \ref{general criterion}, it suffices to prove that there is a $G_K$-equivariant isomorphism $\bb{Q}_\ell^{S_K}\cong \bb{Q}_\ell^{S_k}$ if and only if $\omega\in K$ and there is a $G_K$-equivariant splitting to the short exact sequence
\begin{equation}\label{order 6 p=3 sequence}
    0\to \cal{A}[1-\tau^2]^\circ(\overline{K})\to A[1-\tau^2](\overline{K})\overset{\rho}{\to} \cal{A}[1-\tau^2](\overline{k})\to 0.
\end{equation}

By Proposition \ref{computation summary}, we have $G_K$-equivariant bijections:
\begin{align*}
    S_K & \cong \{1, \pm \omega, \pm \omega^2\} \sqcup (A[1-\tau^2](\overline{K})\setminus 0)/\tau \times \{\omega, \omega^2\} \sqcup (A[2]({\overline{K}})\setminus 0)/\tau,\\
    S_k & \cong \underline{7}\sqcup (D(\cal{A}[1-\tau^2]^\circ)(\overline{k})\times(\cal{A}[1-\tau^2]
    (\overline{k})\setminus 0)
    )/\tau\times \underline{2}\sqcup (A[2](\overline{K})\setminus 0)/\tau.
\end{align*}
Let
\begin{align*}
    S_K' & = (A[1-\tau^2](\overline{K})\setminus \cal{A}[1-\tau^2]^\circ(\overline{K}))/\tau,\\
    S_k' & = (\mathcal{A}[1-\tau^2]^\circ(\overline{K})\times(\cal{A}[1-\tau^2]
    (\overline{k})\setminus 0)
    )/\tau.
\end{align*}

\begin{lem}\label{order 6 p=3 refinement}
    The following are equivalent:
    \begin{enumerate}
        \item there is a $G_K$-equivariant isomorphism $\bb{Q}_\ell^{S_K}\cong \bb{Q}_\ell^{S_k}$;
        \item $\omega\in K$ and there is a $G_K$-equivariant isomorphism $\bb{Q}_\ell^{S_K'}\cong \bb{Q}_\ell^{S_k'}$.
    \end{enumerate}
\end{lem}
\begin{proof}
    First, we may assume $\omega\in K$: otherwise counting the possible fixed points of $G_K$ on $S_K$ and $S_k$ gives a contradiction as in the proof of Lemma \ref{order 3 refinement}.
    Then by Corollary \ref{group scheme embedding special case}, $D(\cal{A}[1-\tau^2]^\circ)(\overline{k})\cong \cal{A}[1-\tau^2]^\circ(\overline{K})$. Therefore,
    \[
        S_k\cong \underline{7}\sqcup (A[2](\overline{K})\setminus 0)/\tau\sqcup S_k'\times \underline{2}.
    \]
    By semi-simplicity, it suffices to prove $S_K\cong \underline{7}\sqcup S_K'\times \underline{2}\sqcup (A[2](\overline{K})\setminus 0)/\tau$.
    
    Note that the two elements in $\cal{A}[1-\tau^2]^\circ(\overline{K})\setminus 0$ form a single $\tau$-orbit preserved by $G_K$. Together with the assumption $\omega\in K$, this gives
    \begin{align*}
        S_K & \cong \underline{5} \sqcup (A[1-\tau^2](\overline{K})\setminus 0)/\tau \times \underline{2} \sqcup (A[2]({\overline{K}})\setminus 0)/\tau\\
        & \cong \underline{7}\sqcup(A[1-\tau^2](\overline{K})\setminus \cal{A}[1-\tau^2]^\circ(\overline{K}))/\tau\times\underline{2} \sqcup (A[2]({\overline{K}})\setminus 0)/\tau\\
        & = \underline{7}\sqcup (A[2](\overline{K})\setminus 0)/\tau\sqcup S_K'\times\underline{2}.
    \end{align*}
\end{proof}

\begin{proof}[Proof of Theorem \ref{main theorem}(4)]
    By Lemma \ref{order 6 p=3 refinement}, it suffices to prove that $\bb{Q}_{\ell}^{S_K'}\cong \bb{Q}_{\ell}^{S_k'}$ as $G_K$-representations if and only if there is a $G_K$-equivariant splitting to \eqref{order 6 p=3 sequence}.

    First, we claim that $\tau$ acts as $-1$ on $A[1-\tau^2](\overline{K})$. This follows since $A[1-\tau^2](\overline{K})$ is a vector space of dimension $2$ over $\bb{F}_3$, and $\tau$ on $A[1-\tau^2](\overline{K})$ is an automorphism of order $2$ that only fixes the origin.

    We now prove the ``only if'' direction. Note that $0\times(\cal{A}[1-\tau^2](\overline{k})\setminus 0)$ gives an element in $S_k'$ fixed by $G_K$. Thus, the only non-trivial $G_K$-action on $S_k'$ exchanges the remaining two elements. Then using the isomorphism $\bb{Q}_{\ell}^{S_K'}\cong \bb{Q}_{\ell}^{S_k'}$, we see that $G_K$ fixes some element in $S_K'$ and there exists a $\tau$-orbit $\{P,\tau(P)\}$ preserved by $G_K$ in $A[1-\tau^2](\overline{K})\setminus \cal{A}[1-\tau^2]^\circ(\overline{K})$. Then $\{0, P,\tau(P)\}$ provides a $G_K$-equivariant splitting to \eqref{order 6 p=3 sequence}.

    Conversely, fix a $G_K$-equivariant isomorphism $A[1-\tau^2](\overline{K})\cong \cal{A}[1-\tau^2]^\circ(\overline{K})\times \cal{A}[1-\tau^2](\overline{k})$. This is also $\tau$-equivariant since $\tau$ acts as $-1$ on $A[1-\tau^2](\overline{K})$. Then we obtain the following $G_K\times\langle\tau\rangle$-equivariant isomorphisms:
    \begin{align*}
        & A[1-\tau^2](\overline{K})\setminus \mathcal{A}[1-\tau^2]^\circ(\overline{K})\\
        \cong{} & (\mathcal{A}[1-\tau^2]^\circ(\overline{K})\times \mathcal{A}[1-\tau^2](\overline{k}))\setminus (\mathcal{A}[1-\tau^2]^\circ(\overline{K})\times 0)\\
        ={} & (\mathcal{A}[1-\tau^2]^\circ(\overline{K})\times (\mathcal{A}[1-\tau^2](\overline{k})\setminus 0))
    \end{align*}
    In particular, taking the $\tau$-orbits yields a $G_K$-equivariant bijection $S_K'\cong S_k'$.
\end{proof}

This finishes the proof of Theorem \ref{main theorem}. 

\begin{rem}
    Let $A$ be an abelian surface over $K$ with ordinary good reduction. By \cite[Remark 6.5]{non-supersingular-kummer}, if $\Kum(A)$ has good reduction, its reduction is also ordinary.

    We extend the argument in \cite[Remark 6.5]{non-supersingular-kummer} to generalized Kummer surfaces. If $X=\Kum(A,\tau)$ has good reduction, then its reduction is $X^\circ=\Kum(\cal{A}_k,\tau)$ by Lemma \ref{find the caonical reduction}. We want to show that $X^\circ$ is an ordinary K3 surface.

    Let $S_k$ be the set of exceptional divisors of $X^\circ_{\overline{k}}\to \cal{A}_{\overline{k}}/\tau$. Let $W=W(\overline{k})$ be the ring of Witt vectors of $\overline{k}$. To show that $X^\circ$ is ordinary, we need to prove that $H^2_{\cris}(X^\circ_{\overline{k}}/W)_{\bb{Q}}$ has a non-zero slope $0$ part.
    
    We have a formula of crystalline cohomology similar to Proposition \ref{cohomology computation}:
    \[
        H^2_{\cris}(X^\circ_{\overline{k}}/W)_{\bb{Q}}\cong H^2_{\cris}(\cal{A}_{\overline{k}}/W)_{\bb{Q}}^{\tau=\text{id}}\oplus W^{S_k}(-1)_{\bb{Q}}.
    \]
    By \cite[Theorem II.5.2]{illusie_deRham_Witt}, the slope-$0$ summand $(H^1_{\cris}(\cal{A}_{\overline{k}}/W)_{\bb{Q}})_{[0]}$ is the dual of the $p$-adic Tate module $V_p(\cal{A}_k)$ of $\cal{A}_k$. By \cite[Proposition 2.3.2]{Alvaro_thesis}, the minimal polynomial of $\tau$ on $V_p(\cal{A}_k)$ and thus on $(H^1_{\cris}(\cal{A}_{\overline{k}}/W)_{\bb{Q}})_{[0]}$ is a cyclotomic polynomial. In particular, $\tau$ acts on $(H^1_{\cris}(\cal{A}_{\overline{k}}/W)_{\bb{Q}})_{[0]}$ as a $2\times 2$-matrix with determinant $1$, and thus trivially on $\wedge^2 (H^1_{\cris}(\cal{A}_{\overline{k}}/W)_{\bb{Q}})_{[0]}\subset H^2_{\cris}(\cal{A}_{\overline{k}}/W)_{\bb{Q}}$. This gives a non-zero subspace of slope $0$ of $H^2_{\cris}(X^\circ_{\overline{k}}/W)_{\bb{Q}}$.
\end{rem}

\begin{rem}
    One might notice that the criterion for the good reduction of $X=\Kum(A,\tau)$ and that of $Y=\Kum(A)$ are the same if $p=2$ and $\tau$ has order $6$. In this remark, we justify this phenomenon by proving directly that the good reduction of $Y$ implies that of $X$.

    Since $\tau^3=-1$, it follows that $\tau$ descends to $A/\tau^3$ and lifts to $Y$. In particular, $\tau$ on $Y$ has order $3$ and is tame. Also, $X$ is the minimal desingularization of $Y/\tau$.

    We claim that $\tau$ on $Y$ is symplectic. Let $U\subset A$ be the complement of $A[2]$, and let $q:A\to A/\tau^3$ be the quotient map. Then $V=q(U)\subset A/\tau^3$ is smooth, and $q|_U$ is finite \'etale. Thus, we obtain
    \[
        H^0(V,\omega_V)= H^0(U,\omega_U)^{\tau^3=\text{id}}=H^0(U,\omega_U).
    \]
    Via the minimal desingularization $Y\to A/\tau^3$, we view $V$ as a dense open subscheme of $Y$, so the restriction $H^0(Y,\omega_Y)\to H^0(V,\omega_V)=H^0(U,\omega_U)$ is injective. This shows that $\tau$ acts trivially on $H^0(Y,\omega_Y)$.
    
    Suppose $Y$ has good reduction. Since $\tau$ on $Y$ is tame and symplectic, by \cite[Corollary 1.9]{extendability_paper}, $Y$ has a smooth model $\cal{Y}$ such that $\tau$ extends to $\cal{Y}$. Take the quotient $\cal{Y}/\tau$. Again by tameness, the singularities on its generic fibre $Y/\tau$ and special fibre $\cal{Y}_k/\tau$ are the same. Then we may desingularize $\cal{Y}/\tau$ to get a smooth model of $X$.

    Similarly, when $\tau$ has order $6$ and $p=3$, $\Kum(A,\tau)$ has good reduction if and only if $\Kum(A,\tau^2)$ has good reduction. As in the $p=2$ case, we can prove directly that the good reduction of $\Kum(A,\tau^2)$ implies that of $\Kum(A,\tau)$.
\end{rem}

\begin{rem}
    The method we use to prove Theorem \ref{main theorem} does not adapt to $A$ with supersingular good reduction. In that case, one can show that not all singularities on $\cal{A}_{\overline{k}}/\tau$ are rational double points, so we cannot apply \cite{artin_simultaneous_res} to prove Theorem \ref{potential good redn}, and our proof of Theorem \ref{main theorem} breaks down.
\end{rem}

\section{Examples}

Now we apply Theorem \ref{main theorem} to some examples of generalized Kummer surfaces.

\subsection{Examples of generalized Kummer surfaces and their good reduction}\label{toy examples}

Let $E$ be an elliptic curve over $K$ with ordinary good reduction. Let $A=E\times E$, which is an abelian surface over $K$ also with ordinary good reduction. Let $\cal{E}$ and $\cal{A}$ be the N\'eron models of $E$ and $A$ respectively. Note that $\cal{A}=\cal{E}\times \cal{E}$.

In this sub-section, for each $n\in \{2,3,4,6\}$, we give an example of $\tau$ on $A$ of order $n$, such that the minimal desingularization $\Kum(A,\tau)$ of $A/\tau$ is a K3 surface. Then we apply Theorem \ref{main theorem} to obtain a criterion for the good reduction of $\Kum(A,\tau)$ in the wild case.

\subsubsection{The order $2$ case}

Let $\tau=-1$. Consider the usual Kummer surface $\Kum(A)=\Kum(A,\tau)$.

\begin{prop}\label{good reduction of order 2 example}
    Suppose $p=2$. Then $\Kum(A)$ has good reduction if and only if $E[2](\overline{K})$ is trivial.
\end{prop}
\begin{proof}
    By Theorem \ref{theorem usual kummer}, $\Kum(A)$ has good reduction if and only if there is a $G_K$-equivariant splitting to
    \[
        0\to \cal{A}[2]^\circ(\overline{K})\to A[2](\overline{K})\to \cal{A}[2](\overline{k})\to 0.
    \]
    Since $\cal{A}[2]=\cal{E}[2]\times \cal{E}[2]$, this is equivalent to a $G_K$-equivariant splitting to
    \begin{equation}\label{order 2 example sequence}
        0\to \cal{E}[2]^\circ(\overline{K})\to E[2](\overline{K})\to \cal{E}[2](\overline{k})\to 0.
    \end{equation}
    Note that $\cal{E}[2]^\circ(\overline{K})$ and $\cal{E}[2](\overline{k})$ both have order $2$. Therefore, $G_K$ acts trivially on $\cal{E}[2]^\circ(\overline{K})$ and $\cal{E}[2](\overline{k})$. Then \eqref{order 2 example sequence} splits if and only if $E[2](\overline{K})$ is trivial.
\end{proof}

\subsubsection{The order $3$ case}

This example is from \cite[Section 7.2]{auto_kummer_surface_kondo}. Let $\tau:E\times E\to E\times E$ defined by
\[
    (P,Q)\mapsto (-P-Q,P).
\]
This is an automorphism of order $3$ and the minimal desingularization of $A/\tau$ is indeed a K3 surface $\Kum(A,\tau)$ by \cite[Section 7.2]{auto_kummer_surface_kondo}.

\begin{prop}\label{good reduction of order 3 example}
    Suppose $p=3$. Then $\Kum(A,\tau)$ has good reduction if and only if $\omega\in K$ and there is a $G_K$-equivariant splitting to
    \[
        0\to \cal{E}[3]^\circ(\overline{K})\to E[3](\overline{K})\to \cal{E}[3](\overline{k})\to 0.
    \]
\end{prop}
\begin{proof}
    By Theorem \ref{main theorem}, $\Kum(A,\tau)$ has good reduction if and only if $\omega\in K$ and there is a $G_K$-equivariant splitting to
    \[
        0\to \cal{A}[1-\tau]^\circ(\overline{K})\to A[1-\tau](\overline{K})\to \cal{A}[1-\tau](\overline{k})\to 0.
    \]
    By calculation, $\cal{A}[1-\tau]\cong \cal{E}[3]$. This finishes the proof.
\end{proof}

\subsubsection{The order $4$ case}

This example is from \cite[Section 3.2.3]{Alvaro_thesis}. Let $\tau:E\times E\to E\times E$ defined by
\[
    (P,Q)\mapsto (-Q,P).
\]
This is an automorphism of order $4$ and the minimal desingularization of $A/\tau$ is indeed a K3 surface $\Kum(A,\tau)$ by \cite[Section 3.2.3]{Alvaro_thesis}.

\begin{prop}\label{good reduction of order 4 example}
    Suppose $p=2$. Then $\Kum(A,\tau)$ has good reduction if and only if $i\in K$ and $E[2](\overline{K})$ is trivial.
\end{prop}
\begin{proof}
    By Theorem \ref{main theorem}, $\Kum(A,\tau)$ has good reduction if and only if $i\in K$, $A[1-\tau](\overline{K})$ is trivial, and there is a $G_K\times\langle\tau\rangle$-equivariant splitting to
    \begin{equation}\label{order 4 example sequence}
        0\to \cal{A}[2]^\circ(\overline{K})\to A[2](\overline{K})\to \cal{A}[2](\overline{k})\to 0.
    \end{equation}
    Note that $\cal{A}[1-\tau]\cong \cal{E}[2]$ and $\cal{A}[2]=\cal{E}[2]\times \cal{E}[2]$. Thus, it suffices to prove that \eqref{order 4 example sequence} splits $G_K\times\langle\tau\rangle$-equivariantly if $E[2](\overline{K})$ is trivial. Suppose $E[2](\overline{K})$ is trivial. Then $A[2](\overline{K})=E[2](\overline{K})\times E[2](\overline{K})$ is trivial. Therefore, we only need to find a $\tau$-equivariant splitting to \eqref{order 4 example sequence}.
    
    We do this explicitly. Let $a\in \cal{E}^\circ[2](\overline{K})$ be the non-zero element. Pick any $b\in E[2](\overline{K})\setminus \cal{E}^\circ[2](\overline{K})$ such that $a,b$ form a basis of $E[2](\overline{K})$ over $\bb{F}_2$. Then $a_1=(a,0), a_2=(0,a), b_1=(b,0), b_2=(0,b)$ form a basis of $A[2](\overline{K})$. Also, $a_1, a_2$ form a basis of $\cal{A}[2]^\circ(\overline{K})$. Then $\langle b_1,b_2\rangle$ provides a $\tau$-equivariant section to \eqref{order 4 example sequence}.
\end{proof}

\subsubsection{The order $6$ case}

This example is from \cite[Section 3.2.4]{Alvaro_thesis}. Let $\tau:E\times E\to E\times E$ defined by
\[
    (P,Q)\mapsto (P+Q,-P).
\]
This is an automorphism of order $6$ and the minimal desingularization of $A/\tau$ is indeed a K3 surface $\Kum(A,\tau)$ by \cite[Section 3.2.4]{Alvaro_thesis}.

\begin{prop}\label{good reduction of order 6, p=2 example}
    Suppose $p=2$. Then $\Kum(A,\tau)$ has good reduction if and only if $E[2](\overline{K})$ is trivial.
\end{prop}
\begin{proof}
    By Theorem \ref{main theorem}, $\Kum(A,\tau)$ has good reduction if and only if $\Kum(A)$ has good reduction. Then this follows from Proposition \ref{good reduction of order 2 example}.
\end{proof}

\begin{prop}\label{good reduction of order 6, p=3 example}
    Suppose $p=3$. Then $\Kum(A,\tau)$ has good reduction if and only if $\omega\in K$ and there is a $G_K$-equivariant splitting to
    \[
        0\to \cal{E}[3]^\circ(\overline{K})\to E[3](\overline{K})\to \cal{E}[3](\overline{k})\to 0.
    \]
\end{prop}
\begin{proof}
    By Theorem \ref{main theorem}, $\Kum(A,\tau)$ has good reduction if and only if $\Kum(A,\tau^2)$ has good reduction. Then this follows from Proposition \ref{good reduction of order 3 example}.
\end{proof}

\subsection{Non-examples}

In this sub-section, we show that the conditions in Theorem \ref{main theorem} are not automatic. For each possible order of $\tau$, we find an elliptic curve $E$ such that, with $A=E\times E$ and the corresponding $\tau$ from Section \ref{toy examples}, $\Kum(A,\tau)$ does not have good reduction --- even if $\omega\in K$ or $i\in K$ (if required by the criteria) and $H^2_{\et}(\Kum(A,\tau)_{\overline{K}},\bb{Q}_\ell)$ is unramified.

Let $S_K$ be the set of exceptional divisors of the minimal desingularization of $A_{\overline{K}}/\tau$. Note that $H^2_{\et}(\Kum(A,\tau)_{\overline{K}},\bb{Q}_\ell)$ is unramified if and only if $S_K$ is unramified by Proposition \ref{cohomology computation}.

\begin{exmp}\label{order 2 counterexample}\cite[Example 6.4]{non-supersingular-kummer}
    Let $K=\bb{Q}_2$ and let $E$ be the elliptic curve over $K$ given by
    \[
        E:y^2+xy=x^3-20x-5.
    \]
    It has good reduction over $K$ and $\cal{E}_k$ is given by
    \[
        \cal{E}_k:y^2+xy=x^3+1.
    \]
    The $2$-torsion points of $\cal{E}_k$ are the origin and $(0,1)$, so $E$ has ordinary good reduction.

    By computation, the $2$-torsion points of $E$ are defined over $\bb{Q}_2(\sqrt{5})$, which is unramified over $\bb{Q}_2$. This gives the desired example in the order $2$, $p=2$ case and the order $6$, $p=2$ case. Also, since $\bb{Q}_2(\sqrt{5})/\bb{Q}_2$ is unramified and $\bb{Q}_2(i)/\bb{Q}_2$ is ramified, $\bb{Q}_2(i,\sqrt{5})/\bb{Q}_2(i)$ is unramified. Then after replacing $K$ with $\bb{Q}_2(i)$, we obtain the example in the order $4$, $p=2$ case.
\end{exmp}

\begin{exmp}[The order $3$, $p=3$ case]\label{order 3 counterexample}
    Let $K=\bb{Q}_3(\omega)$ and let $L/K$ be the unique unramified extension of degree $3$. By Kummer theory, there exists $t\in \cal{O}_K^\times$ such that $L=K(\sqrt[3]{t})$. Let $t'$ be the reduction of $t$ in the residue field $k=\bb{F}_3$.

    Consider the elliptic curve $E$ given by the homogeneous equation
    \[
        E:x^3+y^3+tz^3=xyz
    \]
    with origin $[1:-1:0]$. Then $E$ has good reduction and $\cal{E}_k$ is given by
    \[
        \cal{E}_k:x^3+y^3+t'z^3=xyz.
    \]

    The $3$-torsion points are exactly the points of inflection. By explicit computation, $[x:y:z]$ is a point of inflection if and only if $xyz=0$.
    Then the $3$-torsion points of $E$ and their reductions are given by:
    \begin{align*}
        \rho:E[3](\overline{K}) & \to E[3](\overline{k}),\\
        [1:-\omega^j:0] & \mapsto [1:-1:0],\\
        [\omega^j\sqrt[3]{t}:0:-1] & \mapsto [t':0:-1],\\
        [0:\omega^j\sqrt[3]{t}:-1] & \mapsto [0:t':-1],
    \end{align*}
    where $j\in\{0, 1, 2\}$.
    
    In particular, $E$ has ordinary good reduction, and the $3$-torsion points of $E$ are defined over $L=K(\sqrt[3]{t})$, which is unramified over $K$. This gives the desired example in the order $3$, $p=3$ case and the order $6$, $p=3$ case.
\end{exmp}

\printbibliography

\end{document}